\documentclass[11pt]{amsart}
\usepackage{latexsym}
\usepackage{color}
\usepackage{amsfonts}
\usepackage{amsmath}
\usepackage{parselines}
\usepackage{amssymb}
\usepackage{amsthm}
\usepackage[english]{babel}
\newtheorem{theorem}{Theorem}[section]
\newtheorem*{theorem*}{Theorem}
\newtheorem{corollary}{Corollary}[theorem]
\newtheorem*{corollary*}{Corollary 2.4}
\newtheorem{lemma}[theorem]{Lemma}
\newtheorem{definition}{Definition}
\usepackage[margin=1in]{geometry}
\usepackage{bookmark}
\newcommand{\Mod}[1]{\ (\mathrm{mod}\ #1)}

\begin{document}
 	\title{ONE-LEVEL DENSITY OF ZEROS OF DIRICHLET $L$-FUNCTIONS OVER FUNCTION FIELDS}
 	\author{Hua Lin}
	\maketitle
	
	\begin{abstract}
		We compute the one-level density of zeros of order $\ell$ Dirichlet $L$-functions over function fields $\mathbb{F}_q[t]$ for $\ell=3,4$ in the Kummer setting ($q\equiv1\Mod{\ell}$) and for $\ell=3,4,6$ in the non-Kummer setting ($q\not\equiv1\Mod{\ell}$). In each case, we obtain a main term predicted by Random Matrix Theory (RMT) and lower order terms not predicted by RMT. We also confirm the symmetry type of the families is unitary, supporting Katz and Sarnak's philosophy. 
	\end{abstract}

\section{Introduction}

Hilbert and P\'{o}lya suggested that there is a natural spectral interpretation of the zeros of the Riemann zeta function in the early 1900s. Although there was little evidence at the time, today there are many results supporting their suggestion. Trying to understand the distribution of zeros of the Riemann zeta function, Montgomery computed the pair correlation of normalized zeros $\hat{\gamma}$ \cite{Mon}. Let $Z\left(T\right)$ be the set of zeros $\{\gamma=\sigma + it: \zeta(\gamma)=0, 0\le\sigma<1, 0<t<T\}$. He found that, under the Riemann hypothesis, for a test function $f$ with the support of its Fourier transform $\hat{f}$ in $(-1, 1)$,
\begin{equation*}
	\lim_{T \rightarrow \infty}\frac{1}{\left|Z(T)\right|} \sum_{\substack{\gamma, \gamma^\prime \in Z(T)  \\ \alpha \le \hat{\gamma}-\hat{\gamma^\prime}\le \beta }} f\left(\hat{\gamma}-\hat{\gamma^\prime}\right)
	= \int_{\alpha}^{\beta}f(x)\left(1-\left(\frac{\sin(\pi x)}{\pi x}\right)^2\right)dx.
\end{equation*}   
The physicist Freeman Dyson recognized the function $\displaystyle 1-\left(\frac{\sin(\pi x)}{\pi x}\right)^2$ in the integrand above and pointed out to Montgomery that the eigenvalues of a random complex Hermitian or unitary matrix of large order have precisely the same distribution function \cite{Mon}. Since then, numerous parallels were observed in the statistics of eigenvalues of random matrices and the statistics of zeros of the zeta function and $L$-functions.

In 1999, Katz and Sarnak conjectured that the statistics of zeros in \textit{families} of $L$-functions, in the limit when the conductor of the $L$-functions gets large, follow the distribution laws of classical random matrices \cite{KSa,KSb}. This is known as Katz and Sarnak's philosophy. In their work over finite fields, they used deep equidistribution theorems by Deligne taking the genus $g$ of the curves and $\left|\mathbb{F}_q[x]\right|$ both to infinity. Fixing the ground field and just taking the genus to infinity cannot be similarly studied using the same method; instead, the computations depend more on the arithmetic of the family. 

In this paper, we consider a fixed ground field $\mathbb{F}_q(x)$ and let the genus $g$ tend to infinity. We compute the one-level density of zeros of Dirichlet $L$-functions and obtained results where the main term matches that predicted by RMT and lower order terms given by the Ratios Conjecture. Using ratios of shifted $L$-functions, the Ratios Conjecture gives more precise heuristics including lower order terms for the family \cite{RCNF}, \cite{RCFF}. 

One-level density of zeros studies the average behavior of low-lying zeros of families of $L$-functions. It is a local statistic that differs depending on the symmetry type of the family. Computing the one-level density can lead to results of no-vanishing at the central value which are of special interests. 
A conjecture due to Chowla \cite{Chowla} states that the Dirichlet $L$-function $L(s, \chi)$ does not vanish at the central value $s=1/2$ for any primitive character $\chi$. Over number fields, \"{O}zl\"{u}k and Snyder computed the one-level density of zeros of quadratic Dirichlet $L$-functions assuming the Generalized Riemann Hypothesis (GRH) \cite{OS}; they proved that at least 15/16 of the members in the family do not vanish at the central point $s=1/2$. For the same family, Soundararajan later achieved a remarkable unconditional result of at least 87.5\% non-vanishing using mollified moments \cite{Sound}.

Similar results were shown for the quadratic case over function fields. In the ring $\mathbb{F}_q\left[t\right]$ of an odd prime power $q$, let $\mathcal{H}_{2g+1}$ denote the ensemble of monic, squarefree polynomials of degree $2g+1$ over $\mathbb{F}_q[t]$. Let $\phi\left(\theta\right)=\sum_{|n|\le N}\hat{\phi}\left(n\right)e\left(n\theta\right)$ be a real, even trigonometric polynomial, where $e(x)=e^{2\pi i x}$ and  $\Phi\left(2g\theta\right)=\phi(\theta)$. Rudnick computed the one-level density of quadratic Dirichlet $L$-functions over function fields by analyzing the traces of Frobenius classes of hyperelliptic curves \cite{Rudnick}. He showed that for $\hat{\Phi}$ supported in $\left(-2,2\right)$, when $q$ is fixed and as $g\rightarrow \infty$, 
$$\frac{1}{|\mathcal{H}_{2g+1}|}\sum_{D \in \mathcal{H}_{2g+1}}\sum_{j=1}^{2g}\Phi\left(2g\theta_{j,D}\right)=\hat{\Phi}(0)-\frac{1}{g}\sum_{n\le g}\hat{\Phi}(n/g)+\frac{dev(\Phi)}{g}+o(1/g).$$
Here
$$dev(\Phi)=\hat{\Phi}(0)\sum_{P\in \mathcal{P}}\frac{d(P)}{|P|^2-1}-\frac{\hat{\Phi}(1)}{q-1},$$ 
where the sum is over all monic irreducible polynomials $P$ and $d(P)$ denotes the degree of $P$.
Bui and Florea in \cite{BF} extended Rudnick's result and computed further lower order terms not predicted by any heuristics when the support of the Fourier transform is limited. These one-level density results provide support for Katz and Sarnak's philosophy.

The investigation of higher order characters started more recently. Over number fields, Baier and Young studied moments of cubic $L$-functions and conjectured that their methods will work for moments of quartic and sextic $L$-functions. Cho and Park \cite{ChoandPark} studied the one-level density of cubic $L$-functions in the Kummer setting under GRH and obtained results for test functions $f$ with the support of $\hat{f}$ in $(-1,1)$. The authors found that their results match those predicted by the Ratios Conjecture. David and G\"{u}lo\u{g}lu computed the one-level density of a thin family of cubic Dirichlet $L$-functions and obtained a positive proportion of non-vanishing at $s=1/2$ in the Kummer setting \cite{DG}. Gao and Zhao \cite{GZ1}, \cite{GZ2}, \cite{GZ6} studied various statistics of thin families of quartic and sextic $L$-functions over number fields under GRH, where the latter showed a 2/45 of non-vanishing in the thin sextic family. Over function fields, David, Florea and Lalin studied mean values of cubic $L$-functions and obtained a positive proportion of non-vanishing in the full family \cite{DFL}, \cite{DFL2}. 

In this paper we consider higher order Dirichlet $L$-functions over function fields with characters $\chi$ such that $\chi^{\ell}=1$. Similar to \cite{DFL}, our approach differs depending on the existence of $\ell^{th}$ roots of unity in $\mathbb{F}_q^\times$, i.e., if $q\equiv 1\Mod\ell$.
We compute the one-level density of zeros of cubic and quartic Dirichlet $L$-functions over function fields when $q\equiv1\Mod\ell$ (Kummer setting), and for cubic, quartic and sextic Dirichlet $L$-functions when $q\not\equiv1\Mod\ell$ (non-Kummer setting). In both settings, we obtained a main term matching that predicted by RMT and lower order terms not predicted by RMT. We confirm the symmetry type of the families is unitary. 


In Section \ref{section 2}, we collect the statements of the main theorems in this paper.

In Section \ref{section 3}, first we give a list of notations and cover some relevant background for $L$-functions in function fields. Then we define the primitive Dirichlet characters in both settings. The construction is standard in the Kummer setting, and in the non-Kummer setting, we use a more natural description of primitive characters in $\mathbb{F}_q[t]$ similar to the one used for cubic characters in \cite{DFL}, and was given in \cite{BSM}, where the authors extended the work of \cite{BY} to function fields. Lastly we prove the explicit formula, which allows us to rewrite the sum over the zeros in the one-level density into a sum over prime powers. We also explicitly write down the main term and the error terms in both settings.
 
In Section \ref{section 4}, we compute the one-level density for cubic and quartic Dirichlet $L$-functions in the Kummer setting and prove the symmetry type of the family is unitary.

In Section \ref{section 5} we compute the one-level density for cubic, quartic and sextic Dirichlet $L$-functions in the non-Kummer setting and prove the symmetry type of the family is unitary.\\


\section{Statement of Results} \label{section 2}
We have the following results for the one-level density of zeros of Dirichlet $L$-function over function fields. 
The explicit one-level density formulas are given in \eqref{equation 1LD K3}, \eqref{equation 1LD K4} for the two Kummer cases and \eqref{equation 1LD nK} for the non-Kummer cases.
\begin{theorem}
	\label{Theorem 1}
	Let $\phi(\theta)=\sum_{|n|\le N}\hat{\phi}(n)e(n\theta)$ be any real, even trigonometric polynomial and $\displaystyle\Phi\left(g\theta\right)=\phi(\theta)$. Let $\Sigma_3^\mathrm{K}\left(\Phi, g\right)$ be the one-level density of zeros of cubic Dirichlet $L$-functions in the Kummer setting given in \eqref{1LD K} for $\ell=3$.
	

	We have that
	\begin{equation}
		\label{equation final l=3}
		\begin{split}
			\Sigma_3^\mathrm{K}\left(\Phi, g\right)
			=
			\hat{\Phi}&(0)
			-
			\frac{2}{g}\sum_{1\le n\le N/3}\hat{\Phi}\left(\frac{3n}{g}\right)\sum_{\substack{Q\in \mathcal{P}_{q, n/r}\\r\ge1}}\frac{d(Q)}{|Q|^{3r/2}\left(1+2|Q|^{-1}\right)}\\
			+&4\mathfrak{Re}\left( \frac{h_1}{g}\sum_{1\le n\le N/3}\hat{\Phi}\left(\frac{3n}{g}\right)\sum_{\substack{Q\in \mathcal{P}_{q, n/r}\\r\ge1}}\frac{d(Q)\left(1-c_Q\right)}{|Q|^{3r/2}\left(1+2|Q|^{-1}\right)}\right)\\
			+&
			\frac{2{h_2}}{g}\sum_{1\le n\le N/3}\hat{\Phi}\left(\frac{3n}{g}\right)\sum_{\substack{Q\in \mathcal{P}_{q, n/r}\\r\ge1}}\frac{2d(Q)^2|Q|^{-1}}{|Q|^{3r/2}\left(1+2|Q|^{-1}\right)^2}
			+O\left(q^{N/2}q^{-g/2}q^{\epsilon N}\right),
		\end{split}
	\end{equation}
	where $c_Q$, $h_1$ and $h_2$ are explicitly defined in \eqref{constant c_Q}, \eqref{constant h_1} and \eqref{constant h_2} respectively.
\end{theorem}
\noindent
	We note here that $|c_Q|<2$, and $h_1, h_2$ are of order $1/g$. The sums over $n$ and $Q$ are convergent in each terms. Therefore, when $N<g$, every term except for $\hat{\Phi}(0)$ vanishes as the genus $g$ tends to infinity.

	Using similar methods as Theorem \ref{Theorem 1}, we prove the following for quartic Dirichlet $L$-functions in the Kummer setting.
\begin{theorem}
	\label{Theorem K2}
	Let $\phi(\theta)=\sum_{|n|\le N}\hat{\phi}(n)e(n\theta)$ be any real, even trigonometric polynomial and $\displaystyle\Phi\left(\frac{2g\theta}{3}\right)=\phi(\theta)$. Let $\Sigma_4^\mathrm{K}\left(\Phi, g\right)$ be the one-level density of zeros of quartic Dirichlet $L$-functions in the Kummer setting given in \eqref{1LD K} for $\ell=4$.
	

	We obtain that
	\begin{equation}
		\label{equation final l=4}
		\begin{split}
			&\Sigma_4^\mathrm{K}\left(\Phi, g\right)
			=
			\hat{\Phi}\left(0\right)\\
			&-
			\frac{3}{g}\sum_{1\le n\le N/4}\hat{\Phi}\left(\frac{6n}{g}\right)\sum_{\substack{Q\in \mathcal{P}_{q,n/r}\\r\ge1}}\frac{d(Q)}{|Q|^{2r}\left(1+2|Q|^{-1}\right)}
			- 
			\frac{3s_2}{g}\sum_{1\le n\le N/4}\hat{\Phi}\left(\frac{6n}{g}\right)\sum_{\substack{Q\in \mathcal{P}_{q, n/r}\\r\ge1}}\frac{2d(Q)^2|Q|^{-1}}{|Q|^{2r}\left(1+2|Q|^{-1}\right)^2}\\
			&+O\left(  q^{N/2}q^{-G/2}q^{\epsilon N}\right),
		\end{split}
	\end{equation}
	where the degree of the conductor is $G=\frac{2g}{3}+1$, and $s_2$ is explicitly defined in \eqref{l=4 constant s_2}.
\end{theorem}

\noindent
Note that the constant $s_2$ is of order $1/G$. Therefore, similar to Theorem \ref{Theorem 1}, when $N<G$ every term except for $\hat{\Phi}(0)$ vanishes as the genus $g$ tends to infinity. 

In the non-Kummer setting, we have the following for order $\ell$ Dirichlet $L$-functions for $\ell=3, 4$ and 6.
\begin{theorem}
	\label{Theorem nK}
	Let $\phi(\theta)=\sum_{|n|\le N}\hat{\phi}(n)e(n\theta)$ be any real, even trigonometric polynomial and $\displaystyle\Phi\left(\frac{2g\theta}{\ell-1}\right)=\phi(\theta)$. Let $\Sigma_\ell^\mathrm{nK}\left(\Phi, g\right)$ be the one-level density of zeros of order $\ell$ Dirichlet $L$-functions in the non-Kummer setting for $\ell=3, 4$ and 6 as seen in \eqref{1LD nK}.
	
	
	We have that
	\begin{equation}
		\begin{split}
			\Sigma&_\ell^\mathrm{nK}\left(\Phi, g\right)
			=
			\hat{\Phi}\left(0\right)
			-
			\frac{\ell-1}{g}\sum_{1\le n\le N}\hat{\Phi}\left(\frac{\left(\ell-1\right)n}{2g}\right)q^{-n/2}\\
			&-\frac{\ell-1}{g}
			\sum_{1\le n\le N/\ell}\hat{\Phi}\left(\frac{\ell(\ell-1)n}{2g}\right) \sum_{\substack{Q\in \mathcal{P}_{q, n/r}\\r \ge 1}}\frac{d(Q)}{|Q|_q^{\ell r/2}\left(1+|Q|_q^{-{2/m_Q}}\right)^{m_Q}}
			+O\left(
			q^{N/2}q^{-D\left(\ell \right)/2}q^{\epsilon(N+g)}\right),
		\end{split}
	\end{equation}
	where $m_Q = \gcd\left(d(Q), 2\right)$.
\end{theorem}

Using the theorems above, we have the following result, which shows how the one-level density corresponds to RMT. This supports Katz and Sarnak's philosophy \cite{KSa}, \cite{KSb}.
\begin{corollary*}[Symmetry type of the family] Let $\phi(\theta)=\sum_{|n|\le N}\hat{\phi}(n)e(n\theta)$ be any real, even trigonometric polynomial and $\displaystyle\Phi\left( \left(D(\ell)-2\right)\theta \right)=\phi(\theta)$. For $\displaystyle N<\frac{2g}{\ell-1}$, we have in the Kummer setting for $\ell=3,4$,
	$$\lim_{g\rightarrow \infty}\Sigma_\ell^\mathrm{K}\left(\Phi, g\right)=
	\int_{-\infty}^{\infty}\hat{\Phi}(y)\hat{W}_{U(D(\ell)-2)}(y)dy + o(1),$$
	and in the non-Kummer setting for $\ell=3,4,6$,
	$$\lim_{g\rightarrow \infty}\Sigma_\ell^\mathrm{nK}\left(\Phi, g\right)=
	\int_{-\infty}^{\infty}\hat{\Phi}(y)\hat{W}_{U(D(\ell)-2)}(y)dy + o(1).$$	
	Here $\hat{W}_{U(D(\ell)-2)}(y)=\delta_0(y)$ denotes the one-level scaling density of the group of unitary matrices.
\end{corollary*}
\noindent
This gives us that the symmetry types of the families are unitary.\\
\section{Preliminary} \label{section 3}
\subsection{Notation and background}
\subsubsection{Notations and Perron's formula} \label{notations} We use the following notations throughout this paper. \\Let
\begin{itemize}
	\item $\mathcal{M}_q$ denote the set of monic polynomials in $\mathbb{F}_q[t]$ and $\mathcal{M}_{q,d}$ be those monic polynomials of degree $d$. (Note that $\left|\mathcal{M}_{q,d}\right|=q^d$.)
	\item $\mathcal{P}_q$ be the monic irreducible polynomials in $\mathbb{F}_q[t]$ and $\mathcal{P}_{q,d}$ be those monic irreducible polynomials of degree $d$.
	\item $\mathcal{H}_q$ be the monic squarefree polynomials in $\mathbb{F}_q[t]$ and $\mathcal{H}_{q,d}$ be those monic squarefree polynomials of degree $d$. (Note that for $d\ge2$, we have $\left|\mathcal{H}_{q,d}\right|=q^d\left(1-\frac{1}{q}\right)$.)
	\item $d(f)$ denote the degree of the polynomial $f$.
	\item $\left|f \right|_{q^n}=q^{nd(f)}$ define the norm of $f$ in $ \mathbb{F}_{q^n}[t]$, and we use $\left|f \right|=q^{d(f)}$ if $f\in \mathbb{F}_q[t]$.
\end{itemize}

For $f\in\mathbb{F}_q[t]$, we have the von Mangoldt function
$$\Lambda(f)
=
\begin{cases}
	d(P) \ \ &\text{if} \  f=cP^k\  \text{for some}\  c\in\mathbb{F}_q^\times \  \text{and} \  k\ge1,\\
	0  \ \ \ &\text{otherwise}.
\end{cases}$$
The Prime Polynomial Theorem can be written in the form \cite{Rosen}
$$\sum_{f\in\mathcal{M}_n}\Lambda(f)=q^n.$$

We also recall Perron's formula over $\mathbb{F}_q[t]$, which we use in the latter sections.
\begin{lemma}[Perron's Formula]
	\label{lemma Perron's formula}
	If the generating series $\mathcal{S}(u)=\sum_{f\in\mathcal{M}_q}a(f)u^{d(f)}$ is absolutely convergent in $\left| u\right|\le r<1,$ then 
	$$\sum_{f\in\mathcal{M}_q,d}a(f)
	=
	\frac{1}{2\pi i}\oint_{\left| u\right|= r}\frac{\mathcal{S}(u)}{u^d}\frac{du}{u},$$
	where, as usual, $\oint$ denote the integral over the circle oriented counterclockwise. 
\end{lemma}

\subsubsection{Order $\ell$ Dirichlet $L$-functions over function fields for $\ell=3,4,6$.}
Throughout this paper, we take $q$ to be an odd prime power coprime to 3. For each order, a congruence condition for $q$ is given in Lemma \ref{reciprocity} to ensure reciprocity. Analogous to \cite{DFL} in the cubic case, for $q\equiv1\pmod\ell$ (the Kummer case) we fix an isomorphism $\Omega_\ell$ from the $\ell^{th}$ roots of unity $\mu_\ell \subset \mathbb{C}^\times$ to the $\ell^{th}$ roots of unity in $\mathbb{F}_q^\times$. We also fix an order $\ell$ character $\chi_\ell$ on the group of units $\mathbb{F}_q^\times$ by 
\begin{equation}
	\label{equation chi_ell}
	\chi_\ell(\alpha)=\Omega_\ell^{-1}\left(\alpha^{\frac{q-1}{\ell}}\right).
\end{equation}
A character on $\mathbb{F}_q[t]$ is even if it is the principal character on $\mathbb{F}_q^\times$, and odd otherwise. Thus in the Kummer case, any order $\ell$ character on $\mathbb{F}_q[t]$ falls into $\ell$ classes depending on its restriction to $\mathbb{F}_q^\times$. It can be either the principal character $\chi_0$ or $\chi_\ell^j$ for some integer $1\le j < \ell.$ The character is even in the first case, and odd otherwise.

We will refer to \cite{Rosen} for background on zeta functions over function fields. For Re$(s)>1$, the zeta function of $\mathbb{F}_q[t]$ is defined to be
$$\zeta_q(s):=\sum_{f\in\mathcal{M}_q}\frac{1}{\left|f\right|_q^s}=\prod_{P\in \mathcal{P}_q}\left(1-\frac{1}{\left|P\right|_q^s}\right)^{-1}.$$
Since $\left|\mathcal{M}_{q,d}\right|=q^d,$ we have 
$$\zeta_q(s)=\frac{1}{1-q^{1-s}}.$$
It is sometimes convenient to make the change of variable $u=q^{-s}$, and we denote the zeta function 
$$\mathcal{Z}_q(u):=\zeta_q(s)=\frac{1}{1-qu}.$$

For an order $\ell$ character $\chi$ with conductor $h$, the Dirichlet $L$-function attached to $\chi$ is defined by
$$L_q\left(s, \chi\right):=\sum_{f\in\mathcal{M}_q}\frac{\chi(f)}{\left|f\right|_q^s}.$$
With the change of variable, we have
$$\mathcal{L}_q\left(u,\chi \right):=
L_q\left(s, \chi\right)=\sum_{f\in\mathcal{M}_q}\chi(f)u^{d(f)}=\prod_{\substack{P\in \mathcal{P}_q\\P\nmid h}}\left(1-\chi(P)u^{d(P)}\right)^{-1}.$$

Now let $C$ be a curve of genus $g$ with conductor $h$ over $\mathbb{F}_q(t)$ and let the function field of $C$ be a cyclic degree $\ell$ extension of the base field. From the Weil conjectures, the zeta function of the curve $C$ can be written as
$$\mathcal{Z}_C(u)=\frac{\mathcal{P}_C(u)}{(1-u)(1-qu)},$$
where $\mathcal{P}_C(u)$ is a polynomial of degree $2g$. Furthermore, we have that from \cite{DavidNotes},
when $\chi$ is odd
$$\mathcal{P}_C(u)=\prod_{i=1}^{\ell-1}\mathcal{L}_q(u,\chi^i),$$ 
and when $\chi$ is even
$$\mathcal{P}_C(u)=\left(1-u\right)^{-\ell+1}\prod_{i=1}^{\ell-1}\mathcal{L}_q(u,\chi^i).$$ 
The Riemann Hypothesis for curves over function fields, proved by Weil in 1948, says that all non-trivial zeros of $\mathcal{L}_q(u, \chi)$ are on the circle $\left|u \right|=q^{-\frac{1}{2}}.$ Let $e(x)=e^{2\pi i x}$. Hence we can express the $L$-function in terms of its zeros
\begin{equation}
	\label{L-function zeros}
	\mathcal{L}_q(u,\chi)=\left(1-u\right)^b\prod_{j=1}^{d(h)-1-b}\left( 1-u\sqrt{q}e(\theta_{j,h}) \right),
\end{equation}
where $b=1$ if $\chi$ is even, and $b=0$ otherwise. \\
Let $\displaystyle D(\ell)=\frac{2g+2\ell-2}{\ell-1}$. Then, using the Riemann-Hurwitz formula,  we have
\begin{equation}
	\label{equation D(l)}
	d(h) 
	=
	\begin{cases}
		D(\ell) \ \ &\chi \  \text{even},\\
		D(\ell)-1	\ \ &\chi \  \text{odd}.
	\end{cases}
\end{equation}

We use the following results from \cite{DFL} in latter sections to bound the $L$-functions. Note that the results hold for characters of any order.
\begin{lemma}
	\label{Lindelof upper bound}
	Let $\chi$ be a primitive order $\ell$ character of conductor $h$ defined over $\mathbb{F}_q[t]$. Then for Re$(s)\ge1/2$ and for all $\epsilon>0$,
	$$\left| L_q(s, \chi ) \right|\ll q^{\epsilon d(h)}.$$
\end{lemma}
\begin{lemma}
	\label{Lindelof lower bound}
	Let $\chi$ be a primitive order $\ell$ character of conductor $h$ defined over $\mathbb{F}_q[t]$. Then for Re$(s)\ge1$ and for all $\epsilon>0$,
	$$\left| L_q(s, \chi ) \right|\gg q^{-\epsilon d(h)}.$$
\end{lemma}

``Completed" $L$-functions also have the functional equation for primitive characters $\chi$ over function fields \cite{Tamam}.
Let $\Lambda_q(u, \chi):= (1-bu)^{-1}\mathcal{L}_q(u, \chi)$ denote the ``completed" $L$-function where $b$ is defined as above. Then
\begin{equation}
	\label{completed L-function}
	\Lambda_q(u, \chi)=\omega(\chi)\left(\sqrt{q}u\right)^{d(h)-1-b}\Lambda_q\left(\frac{1}{qu}, \overline{\chi} \right),
\end{equation}
where $|\omega(\chi)|=1$. We will use this in Lemma \ref{explicit formula lemma} computing the explicit formula.
\subsection{Primitive order $\ell$ characters for $\ell=3, 4, 6$}
Recall that $q$ is coprime to 2 and 3 throughout this paper. 
\subsubsection{The Kummer Setting}
First we describe the primitive order $\ell$ characters $\chi$ in the Kummer setting for a prime $\ell$. We define the $\ell^{th}$ residue symbol for a prime $P$ when $\ell$ divides $q-1$: 
\begin{definition}
	Let $P \in \mathcal{P}_{q}$ and $f\in \mathbb{F}_q[t]$.
	The $\ell^{th}$ Jacobi symbol $\displaystyle \left(\frac{f}{P}\right)_\ell$ is the unique element of $\mathbb{F}_{q^{}}^\times$ such that
	$$\left(\frac{f}{P}\right)_\ell \equiv f^{\frac{|P|-1}{\ell}} \Mod {P}.$$
\end{definition}
\noindent
Note that $\displaystyle \left(\frac{f}{P}\right)_\ell$ is an $\ell^{th}$ roots of unity in $\mathbb{F}_{q^{}}^\times$. We extend this definition multiplicatively to any monic polynomial $F\in \mathcal{M}_q$ in the usual way.

Now we define primitive order $\ell$ characters with conductor $P$ in the Kummer setting:
\begin{definition}
	\label{definition 2}
	Let $\Omega_\ell$ be a fixed isomorphism from the $\ell^{th}$ roots of unity $\mu_\ell\subseteq\mathbb{C}^\times$ to the $\ell^{th}$ roots of unity in $\mathbb{F}_{q^{}}^\times.$ We define $\chi_P(f)=0$ if $P \mid f$, and otherwise $$\chi_P(f)=\Omega_\ell^{-1}\left(\left(\frac{f}{P}\right)_\ell\right).$$ 
\end{definition}
\noindent
For a monic polynomial $H=P_1^{e_1}\cdots P_s^{e_s}$ with distinct primes $P_i$,  we define 
$$\chi_H =\chi_{P_1}^{e_1}\cdots \chi_{P_s}^{e_s},$$
where $\chi_H$ is a primitive character with $\chi_H^\ell=1$ and conductor $F=P_1\cdots P_s$ if and only if $1\le e_i \le \ell-1$ for all $i$.
For $\ell=3$, 
$e_i$ is either $1$ or $2$ as in \cite{DFL}.
Grouping by the exponents, we can thus express $\chi_H$ as $\chi_{F_1}\chi_{F_2}^2\cdots\chi_{F_{\ell-1}}^{\ell-1}$ where $H=F_1F_2^2\cdots F_{\ell-1}^{\ell-1}$, and the $F_i$ are monic squarefree polynomials and pairwise coprime. 
Thus given a conductor $F=F_1F_2\cdots F_{\ell-1}$, we have the corresponding primitive character $\chi_{F_1}\chi_{F_2}^2\cdots\chi_{F_{\ell-1}}^{\ell-1}$.

For $\alpha\in\mathbb{F}_q^\times,$ we have
\begin{equation*}
	\chi_{F_1{F_2}^2\cdots F_{\ell-1}^{\ell-1}}\left( \alpha\right)=\Omega_\ell^{-1}\left( \alpha^{\frac{q-1}{\ell}\left(d_1+2d_2+\dots +(\ell-1)d_{\ell-1}\right)} \right)
\end{equation*}
where $d_i=d\left(F_i\right)$. Thus this character is even if and only if $d_1+2d_2+\dots +(\ell-1)d_{\ell-1}\equiv 0 \pmod \ell$. We have the following:
\begin{equation*}
	d(F)=\begin{cases}
		D\left(\ell\right) \ \  &d_1+2d_2+\dots +(\ell-1)d_{\ell-1}\equiv 0 \Mod \ell,\\
		D\left(\ell\right)-1   \  \ &d_1+2d_2+\dots +(\ell-1)d_{\ell-1} \not\equiv 0 \Mod \ell.
	\end{cases}
\end{equation*}
For convenience, we restrict to the case of odd primitive characters whose restriction to $\mathbb{F}_q^\times$ is $\chi_{\ell}$ as seen \eqref{equation chi_ell}. This means $d_1+2d_2+\dots +(\ell-1)d_{\ell-1}\equiv 1 \pmod \ell$ and $d\left(F\right)=D\left(\ell\right)-1$. 

When $\ell=4,$ we consider the curve given by the affine model 
$$Y^4=F_1F_3^3$$ 
as seen in \cite{Meisner}, where the restriction to $\chi_4$ gives $d_1+3d_3\equiv1\Mod4$. In this case, the square of primitive quartic characters $\chi^2$ remains primitive and the Riemann-Hurwitz formula gives us that the degree of conductors is $D\left(4\right)-1$.

We have thus showed the following.
\begin{lemma}\label{lemma K main term}
	Let $q$ be a prime power coprime to 6, $f\in\mathbb{F}_q[t]$ a monic polynomial, and $\chi_\ell$ as defined in \ref{equation chi_ell}. In the Kummer setting, for $\ell=3$,
	\begin{equation*}
		\sum_{\substack{\chi \ \text{primitive cubic}\\  \text{genus}\left(\chi\right)=g\\\chi \mid_{\mathbb{F}_q^\times}=\chi_{3}}} \chi(f)
		=
		\sum_{\substack{ d_1+d_2=D\left(3\right)-1 \\ d_1+2d_3\equiv1\Mod 3}} \sum_{\substack{ F_1\in \mathcal{H}_{q,d_1} \\ (F_1,f)=1  }}\chi_{F_{1}}(f)
		\sum_{\substack{ F_{2}\in \mathcal{H}_{q,d_{2}} \\ (F_{2},F_1 f)=1  }}\chi_{F_{2}}^{2}(f).
	\end{equation*}
	For $\ell=4$,
	\begin{equation*}
		\sum_{\substack{\chi \ \text{primitive quartic}\\ \chi^2 \text{primitive} \\ \text{genus}\left(\chi\right)=g\\\chi \mid_{\mathbb{F}_q^\times}=\chi_{4}}} \chi(f)
		=
		\sum_{\substack{ d_1+d_3=D\left(4\right)-1 \\ d_1+3d_3\equiv 1\Mod 4}} \sum_{\substack{ F_1\in \mathcal{H}_{q,d_1} \\ (F_1,f)=1  }}\chi_{F_{1}}(f)
		\sum_{\substack{ F_{3}\in \mathcal{H}_{q,d_{3}} \\ (F_{3},F_1 f)=1  }}\chi_{F_{3}}^{3}(f).
	\end{equation*}
\end{lemma} 
\noindent
	We let 
	\begin{equation}
		\label{notation K fam}
		\mathcal{C}_\ell^\mathrm{K}\left(g\right)
	\end{equation}
	denote the family of Dirichlet $L$-functions attached to the characters in the sums above for $\ell\in\{3, 4\}$. (e.g. $\mathcal{C}_3^\mathrm{K}\left(g\right)$ is the family of cubic Dirichlet $L$-functions with character $\chi$ of genus $g$ and $\chi \mid_{\mathbb{F}_q^\times}=\chi_{3}$.) We sometimes use the notation $\displaystyle {\sum_{F\in\mathcal{C}^\mathrm{K}_\ell(g)}} $ to abbreviate the character sums above.
	
Finally we have the reciprocity law derived from Theorem 3.5 in \cite{Rosen}.
\begin{lemma}[Order $\ell$ Reciprocity]
	\label{reciprocity}
	Let $a,b\in \mathcal{M}_q$ be relatively prime polynomials, and let $\chi_a$ and $\chi_b$ be the order $\ell^{th}$ characters defined above. If $q\equiv1\pmod{2\ell}$, then
	$$\chi_a(b)=\chi_b(a).$$
\end{lemma}
\noindent
We let $q\equiv1\pmod{2\ell}$ for $\ell=3,4$ in the Kummer setting and $\ell=3,4,6$ in the non-Kummer setting.  

\subsubsection{The non-Kummer Setting}

We will introduce a general order $\ell$ construction in the non-Kummer setting when $q \not \equiv 1 \Mod{\ell}$. 

Given the congruence condition above, the multiplicative order of $q$ modulo $\ell$ is 2. If $\chi_P$ is an order $\ell$ character of conductor $P$, then 
$$\ell\mid |P|-1 \implies 2 \mid d(P).$$ 
This condition on the degrees implies that $P$ splits into 2 primes, $\displaystyle P= \pi_1\pi_2$, over $\mathbb{F}_{q^{2}}[t]$ \cite{BSM}. Thus, we may define the $\ell^{th}$ Jacobi symbol and the order $\ell$ primitive characters over $\mathbb{F}_{q^2}[t]$ with conductors $\pi_1, \pi_2$ the same way as in the Kummer setting since $q^2\equiv 1\Mod\ell$. Note that $\Omega_\ell$ is now an isomorphism from $\mu_\ell \subseteq \mathbb{C}^\times$ to the $\ell^{th}$ roots of unity in $\mathbb{F}_{q^{2}}^\times$.   

For the Euler-totient function $\phi$, if $\phi\left(\ell\right)=2$, the primitive characters in the non-Kummer setting are the restrictions of some characters of $\mathbb{F}_{q^{2}}[t]$ down to $\mathbb{F}_q[t]$. Let $P\in\mathcal{P}_q$ and $2 \mid d(P)$.  We can define the characters $\{\chi_{\pi_i}\}$ over $\mathbb{F}_{q^{2}}[t]$ as in Definition \ref{definition 2} since $q^{2}\equiv 1 \Mod{\ell}$.\\
By definition, for some $\pi \mid P$ in $\mathbb{F}_{q^{2}}[t]$,
\begin{equation*}
	f^{\frac{q^{d(P)}-1}{\ell}} \Mod{\pi} \equiv \left(\frac{f}{P}\right)_\ell 
	=
	\left(\frac{f}{\pi}\right)_\ell.
\end{equation*} 
Therefore, the restriction $\{\chi_{\pi_i}\}|_{\mathbb{F}_q[t]}$ is exactly the set of primitive order $\ell$ characters $\{\chi_P^{e_i}\}$ for $e_i$ coprime to $\ell$. Note that it is crucial that $\left|\{ \chi_{\pi_i} \} \right|=2=\phi(\ell)=\left|\{\chi_P^{e_i}\}\right|.$ 

For an example of the cubic case over function fields, see \cite{DFL}. The authors referenced the work of  Bary-Soroker and Meisner \cite{BSM} who generalized the work in \cite{BY} over number fields to function fields.
 
For $\ell=3,4,6$ in the non-Kummer setting, we have $q\not\equiv1 \pmod \ell$. Hence, extending multiplicatively to general conductors $H$, the primitive order $\ell$ characters $\chi_H$ where $\chi_H^d$ remains primitive for proper nontrivial divisor $d$ of $\ell$ are given by conductors $F\in\mathbb{F}_{q^{2}}[t]$, $F$ square-free and not divisible by a prime of $\mathbb{F}_{q}[t]$. For example, in the quartic case, these characters were considered in \cite{GZ2}.

These primitive characters of genus $g$ are even 
since $\displaystyle \left|\mathbb{F}_q^\times \right|= q-1 \mid \left(q^{2}-1\right)/\ell$. 
Thus the degree of the conductor $F\in\mathbb{F}_{q^{2}}[t]$ is $D\left(\ell \right)/2$, given by \eqref{equation D(l)}. This gives the following lemma.

\begin{lemma} \label{lemma nK main term} Let $f\in\mathcal{M}_q$. For $\ell=3$,
	$$\sum_{\substack{\chi \  \text{primitive cubic}\\ \text{genus}(\chi)=g }}\chi(f) 
	=
	\sum_{\substack{F\in \mathcal{H}_{q^{2},D\left(3 \right)/2}\\P\mid F \implies P\notin \mathbb{F}_{q^{}}[t] }}\chi_F(f).$$
	For $\ell=4$, we have
	$$\sum_{\substack{\chi \  \text{primitive order 4}\\ \text{genus}(\chi)=g \\ \chi^{2} \text{primitive} }}\chi(f) 
	=
	\sum_{\substack{F\in \mathcal{H}_{q^{2},D\left(4 \right)/2}\\P\mid F \implies P\notin \mathbb{F}_{q^{}}[t] }}\chi_F(f).$$
	For $\ell=6$
	$$\sum_{\substack{\chi \  \text{primitive order $6$}\\ \text{genus}(\chi)=g \\ \chi^{2},\chi^{3}\  \text{primitive} }}\chi(f) 
	=
	\sum_{\substack{F\in \mathcal{H}_{q^{2},D\left(6 \right)/2}\\P\mid F \implies P\notin \mathbb{F}_{q^{}}[t] }}\chi_F(f).$$
\end{lemma}
\noindent
We let 
\begin{equation}
	\label{notation nK fam}
	\mathcal{C}_\ell^\mathrm{nK}\left(g\right)
\end{equation}
denote the family of Dirichlet $L$-functions attached to the characters in the sums above for $\ell\in\{3, 4, 6\}$. (e.g. $\mathcal{C}_4^\mathrm{K}\left(g\right)$ is the family of quartic Dirichlet $L$-functions with character $\chi$ of genus $g$ and $\chi^2$ remains primitive.)
\subsection{The explicit formula, main terms and error terms}
\subsubsection{The Explicit Formula}

	\begin{lemma} [The Explicit Formula]
	\label{explicit formula lemma}
	Let $\chi_F, \overline{\chi_F}$ be Dirichlet characters of conductor $F\in\mathbb{F}_q[t]$ and $\theta_{j,F}$ be the angle of a non-trivial zero of the $L$-function attached to $\chi_F$ as seen in \eqref{L-function zeros}. Recall that we let $e(x)=e^{2\pi ix}$. For $n \ge 0$ and $f\in \mathcal{M}_{q,n}$ we have
	\begin{equation}
		\label{explicit formula equation}
		-\sum_{j=1}^{D}e(n\theta_{j,F})
		=\frac{b}{q^{n/2}}+\sum_{f\in \mathcal{M}_{q,n}}\frac{\Lambda(f)\chi_F(f)}{|f|^{1/2}},
	\end{equation}
	and for $n < 0$ and $f\in \mathcal{M}_{q,|n|}$, we have
	\begin{equation}
		\label{explicit formula equation negative}
		-\sum_{j=1}^{D}e(n\theta_{j,F})
		=\frac{b}{q^{|n|/2}}+\sum_{f\in \mathcal{M}_{q,|n|}}\frac{\Lambda(f)\overline{\chi_F(f)}}{|f|^{1/2}},
	\end{equation}
	where $D=d(F)-1-b$, in which $b=1$ if $\chi_F$ is even and $b=0$ if $\chi_F$ is odd. 
\end{lemma}
\begin{proof}
	We compute the explicit formula for $n\ge0$ first. 
	
	In terms of its zeros
	$$\mathcal{L}_q(u, \chi_F)=\left(1-u\right)^b\prod_{j=1}^{D}\left( 1-u\sqrt{q}e\left(\theta_{j,F}\right) \right),$$
	for some integer $b\in\{0,1\}$, where $b=1$ if $\chi_F$ is an even character, and 0 if $\chi_F$ is odd. $D=d(F)-1-b$ where $F$ is the conductor of $\chi_F$.\\ 		
	Alternatively, we can express the $L$-functions as
	$$\mathcal{L}_q(u, \chi_F)=\prod_{P\in \mathcal{P}_q}\left(1-\chi_F(P)u^{d(P)}\right)^{-1}.$$
	Using log differentiation, we obtain
	$$\frac{-b}{1-u}+ \sum_{j=1}^{D}\frac{-\sqrt{q}e(\theta_{j,F})}{1-u\sqrt{q}e(\theta_{j,F})}
	=
	\sum_{P\in \mathcal{P}_q}
	\frac{d(P)\chi_F(P)u^{d(P)-1}}{1-\chi_F(P)u^{d(P)}}.$$
	Now, expanding the denominators using geometric series, we have
	$$\sum_{n=0}^{\infty}\left[ -b+\sum_{j=1}^{D}-\sqrt{q}e(\theta_{j,F})\left(\sqrt{q}e(\theta_{j,F})\right)^n \right]u^n
	=\sum_{n=0}^{\infty}\sum_{P\in \mathcal{P}_q}d(P)\chi^{n+1}_F(P)u^{(n+1)d(P)-1},$$
	and using the von Mangoldt function, the sum over $P$ on the right can be written as
	$$\sum_{\substack{f = P^{n+1}\\P\in \mathcal{P}_q}}\Lambda(f)\chi_F(f)u^{d(f)-1}.
	$$
	Hence we have 
	$$\sum_{n=0}^{\infty}\left[ -b+\sum_{j=1}^{D}-\sqrt{q}e(\theta_{j,F})\left(\sqrt{q}e(\theta_{j,F})\right)^n \right]u^n
	=
	\sum_{n=0}^{\infty}\sum_{f\in \mathcal{M}_{q,{n+1}}}\Lambda(f)\chi_F(f)u^{n}.$$
	We match the corresponding terms with the same power on $u$ on each side of the equation above, which gives
	\begin{equation*}
		-\sum_{j=1}^{D}e(n\theta_{j,F})
		=\frac{b}{q^{n/2}}+\sum_{f\in \mathcal{M}_{q,n}}\frac{\Lambda(f)\chi_F(f)}{|f|^{1/2}}.
	\end{equation*}
	
	Now for $n< 0$, using the functional equation in \eqref{completed L-function}, we see that if $q^{-1/2}e(-\theta_{j,F})$ are the zeros of $\mathcal{L}_q(u,\chi_F)$, then $q^{-1/2}e(\theta_{j,F})$ are the zeros of $\mathcal{L}_q(u,\overline{\chi_F})$. Thus using the computation above, we derive that for $n<0$
	\begin{equation*}
		-\sum_{j=1}^{D}e(n\theta_{j,F})
		=\frac{b}{q^{|n|/2}}+\sum_{f\in \mathcal{M}_{q,|n|}}\frac{\Lambda(f)\overline{\chi_F(f)}}{|f|^{1/2}}.
	\end{equation*}
\end{proof}

Let $\phi\left(\theta\right)=\sum_{|n|\le N}\hat{\phi}\left(n\right)e\left(n\theta\right)$ be a real, even trigonometric polynomial, and let $\Phi\left(\left(D(\ell)-2\right)\theta\right)=\phi(\theta)$. Using Lemma \ref{lemma K main term} and Lemma \ref{lemma nK main term} we write down the one-level density of zeros we wish to compute in both settings.
\subsubsection{The Kummer Case}
Let 
$\left|\mathcal{C}^{\mathrm{K}}_\ell\left(g\right)\right|$ denote the size of the family of order $\ell$ Dirichlet $L$-functions in the Kummer setting. Recall that we restrict to the case when $\chi$ is odd and $\chi \mid_{\mathbb{F}_q^\times}=\chi_\ell$ for convenience. By \eqref{equation D(l)}, each $L$-function has $D(\ell)-2$ non-trivial zeros. Thus using Lemma \ref{lemma K main term}, the one-level density of Dirichlet $L$-functions $\mathcal{L}_q\left(u, \chi\right)$ with $\chi$ of genus $g$ and $\chi^\ell=1$ over function fields is the sum
\begin{equation}
	\label{1LD K}
	\Sigma_\ell^\mathrm{K}\left( \Phi,g \right)
	=
	\frac{1}{\left|\mathcal{C}^{\mathrm{K}}_\ell\left(g\right)\right|}{\sum_{\substack{ F\in\mathcal{C}_\ell^\mathrm{K}(g) \\  }}}\sum_{j=1}^{D(\ell)-2}\Phi\left((D(\ell)-2)\theta_{j,F}\right),
\end{equation}
where $\displaystyle {\sum_{F\in\mathcal{C}_\ell^\mathrm{K}(g)}}$ denote the sum in Lemma \ref{lemma K main term}.

Since 
$$\Phi\left((D(\ell)-2)\theta_{j,F}\right)
=\frac{1}{D(\ell)-2}\sum_{|n|\le N}\hat{\Phi}\left(\frac{n}{D(\ell)-2}\right)e(n\theta_{j,F}),$$
the sum over the zeros can be written as
\begin{equation*}
	\begin{split}
		\sum_{j=1}^{D(\ell)-2}\Phi( (D(\ell)-2)\theta_{j,F})
		=&\hat{\Phi}(0)+\frac{1}{D(\ell)-2}\sum_{0<|n|\le N}\hat{\Phi}\left(\frac{n}{D(\ell)-2}\right)\sum_{j=1}^{D(\ell)-2}e(n\theta_{j,F})\\
		=&
		\hat{\Phi}(0)-\frac{1}{D(\ell)-2}\sum_{1\le n\le N}\hat{\Phi}\left(\frac{n}{D(\ell)-2}\right)\sum_{f\in \mathcal{M}_{q, n}}\frac{\Lambda(f)}{|f|^{1/2}} \left[\chi_F(f) + \overline{\chi_F(f)} \right],
	\end{split}
\end{equation*}
in which the second equality follows from \eqref{explicit formula equation} and \eqref{explicit formula equation negative} for $b=0$ and $D=D(\ell)-2$.

Thus we write the one-level density in \eqref{1LD K} as
\begin{equation}
	\label{new 1LD K}
	\Sigma_\ell^\mathrm{K}\left( \Phi,g \right)
	=
	\hat{\Phi}(0)-\mathcal{A}^{\mathrm{K}}_{\ell}\left(\Phi, D(\ell)-2\right),
\end{equation}
where we let
\begin{equation}
	\label{equation A Kummer}
	\mathcal{A}^{\mathrm{K}}_{\ell}\left(\Phi, D(\ell)-2\right)=\frac{1}{(D(\ell)-2)|\mathcal{C}^{\mathrm{K}}_\ell(g)|}\sum_{1\le n\le N}\hat{\Phi}\left(\frac{n}{D(\ell)-2}\right)\sum_{f\in \mathcal{M}_{q,n}}\frac{\Lambda(f)}{|f|^{1/2}}{\sum_{\substack{F\in\mathcal{C}_\ell^\mathrm{K}(g)\\ }}}\left[\chi_F(f) + \overline{\chi_F(f)} \right].
\end{equation}
We decompose $\mathcal{A}^{\mathrm{K}}_{\ell}\left(\Phi, D(\ell)-2\right)$ as the sum $M^{\mathrm{K}}_{\ell}\left(\Phi, D(\ell)-2\right)+E^{\mathrm{K}}_{\ell}\left(\Phi, D(\ell)-2\right),$ 
where the main term comes from when $f$ is an $\ell^{th}$ power
\begin{equation}
	\label{main term K}
	\begin{split}
		&M^{\mathrm{K}}_{\ell}\left(\Phi, D(\ell)-2\right)\\
		&=
		\frac{1}{(D(\ell)-2)|\mathcal{C}^{\mathrm{K}}_\ell\left(g\right)|}\sum_{1\le n\le N/\ell}\hat{\Phi}\left(\frac{\ell n}{D(\ell)-2}\right)\sum_{\substack{Q\in \mathcal{P}_{q, n/r}\\r\ge1}}\frac{d(Q)}{|Q|^{\ell r/2}}{\sum_{\substack{F\in\mathcal{C}_\ell^\mathrm{K}(g) }}}\left[\chi_F\left(Q^{\ell r}\right)+\overline{\chi_F\left(Q^{\ell r}\right)}\right],
	\end{split}
\end{equation}
and the non-$\ell^{th}$ power contribution is
\begin{equation}
	\label{error term K}
	\begin{split}
		&E^{\mathrm{K}}_{\ell}\left(\Phi, D(\ell)-2\right)\\
		&=
		\frac{2}{(D(\ell)-2)|\mathcal{C}^{\mathrm{K}}_\ell\left(g\right)|}\sum_{1\le n\le N}\hat{\Phi}\left(\frac{n}{D(\ell)-2}\right)\sum_{\substack{f\in \mathcal{M}_{q,n}\\f \ \text{non-$\ell^{th}$ power}}}\frac{\Lambda(f)}{|f|^{1/2}}{\sum_{\substack{F\in\mathcal{C}_\ell^\mathrm{K}(g) }}}\left[\chi_F(f)+\overline{\chi_F(f)}\right].
	\end{split}
\end{equation}
These equations hold for $\ell=3, 4$ and $q$ an odd prime power coprime to $\ell$. In Section \ref{section 4}, we consider cubic and quartic Dirichlet $L$-functions in the Kummer setting with $q$ coprime to $\ell$ and $q\equiv1\Mod{2\ell}$ satisfying reciprocity.
\subsubsection{The non-Kummer Case}
For $\ell=3,4,6$, let $\left|\mathcal{C}_\ell^\mathrm{nK}(g) \right|$ denote the size of the family of Dirichlet $L$-functions attached to the characters given in Lemma \ref{lemma nK main term} in the non-Kummer setting. Since the characters are even, \eqref{equation D(l)} implies that the associated $L$-functions have $D(\ell)-2$ non-trivial zeros and a trivial zero at $u=1$. Thus the one-level density of order $\ell$ Dirichlet $L$-functions $\mathcal{L}_q\left(u, \chi\right)$ with $\chi$ as in Lemma \ref{lemma nK main term} over function fields is the sum
\begin{equation}
	\label{1LD nK}
	\Sigma_\ell^\mathrm{nK}\left( \Phi,g \right)
	=
	\frac{1}{\left|\mathcal{C}^{\mathrm{nK}}_\ell\left(g\right)\right|}\sum_{\substack{F\in \mathcal{H}_{q^2, D\left(\ell \right)/2}\\P|F \Rightarrow P \notin \mathbb{F}_q[t]}}\sum_{j=1}^{D\left(\ell\right)-2}\Phi\left(\left(D\left(\ell\right)-2\right)\theta_{j,F}\right).
\end{equation}

Since
$$\Phi\left((D\left(\ell\right)-2)\theta_{j,F}\right)
=\frac{1}{D\left(\ell\right)-2}\sum_{|n|\le N}\hat{\Phi}\left(\frac{n}{D\left(\ell\right)-2}\right)e(n\theta_{j,F}),$$
the sum over the zeros can be written as
\begin{equation*}
	\begin{split}
		\sum_{j=1}^{D\left(\ell\right)-2}\Phi( (D\left(\ell\right)-2)\theta_{j,F})
		=&\hat{\Phi}(0)+\frac{1}{D\left(\ell\right)-2}\sum_{0<|n|\le N}\hat{\Phi}\left(\frac{n}{D\left(\ell\right)-2}\right)\sum_{j=1}^{D\left(\ell\right)-2}e(n\theta_{j,F})\\
		=&\hat{\Phi}(0)-\frac{1}{D\left(\ell\right)-2}\sum_{1\le n\le N}\hat{\Phi}\left(\frac{n}{D\left(\ell\right)-2}\right)\sum_{f\in \mathcal{M}_{q,n}}\frac{\Lambda(f)}{|f|^{1/2}}\left[\chi_F(f)+\overline{\chi_F(f)}\right]\\
		&-\frac{2}{D\left(\ell\right)-2}\sum_{1\le n\le N}\hat{\Phi}\left(\frac{n}{D\left(\ell\right)-2}\right)q^{-n/2},
	\end{split}
\end{equation*}
in which the second equality follows from $\eqref{explicit formula equation}$ and \eqref{explicit formula equation negative} for $b=1$ and $D=D(\ell)-2$.\\
Thus we write the one-level density in \eqref{1LD nK} as
\begin{equation}\label{new 1LD nK}
	\Sigma_\ell^\mathrm{nK}\left( \Phi,g \right)
	=
	\hat{\Phi}(0)-\mathcal{A}^{\mathrm{nK}}_{\ell}\left(\Phi, D\left(\ell\right)-2\right)-\frac{2}{D\left(\ell\right)-2}\sum_{1\le n\le N}\hat{\Phi}\left(\frac{n}{D\left(\ell\right)-2}\right)q^{-n/2},
\end{equation}
where we let
\begin{equation}
	\label{equation A nKummer}
	\begin{split}
		&\mathcal{A}^{\mathrm{nK}}_{\ell}(\Phi, D\left(\ell\right)-2)\\
		&=
		\frac{1}{\left(D\left(\ell\right)-2\right)\left|\mathcal{C}^{\mathrm{nK}}_\ell\left(g\right)\right|}\sum_{1\le n\le N}\hat{\Phi}\left(\frac{n}{D\left(\ell\right)-2}\right)\sum_{f\in \mathcal{M}_{q,n}}\frac{\Lambda(f)}{|f|^{1/2}}\sum_{\substack{F\in \mathcal{H}_{q^2, D\left(\ell \right)/2}\\P|F \Rightarrow P \notin \mathbb{F}_q[t]}}\left[\chi_F(f)+\overline{\chi_F(f)}\right].
	\end{split}
\end{equation}
We decompose $\mathcal{A}^{\mathrm{nK}}_{\ell}(\Phi,D\left(\ell\right)-2)$ as the sum $M^{\mathrm{nK}}_{\ell}(\Phi,D\left(\ell\right)-2)+E^{\mathrm{nK}}_{\ell}(\Phi,D\left(\ell\right)-2),$ 
where the main term comes from when $f$ is an $\ell^{th}$ power
\begin{equation}
	\label{main term nK}
	\begin{split}
		&M^{\mathrm{nK}}_{\ell}(\Phi, D\left(\ell\right)-2)\\
		&=
		\frac{1}{\left(D\left(\ell\right)-2\right)\left|\mathcal{C}^{\mathrm{nK}}_\ell\left(g\right)\right|}\sum_{1\le n\le N/\ell}\hat{\Phi}\left(\frac{\ell n}{D\left(\ell\right)-2}\right)\sum_{\substack{Q\in \mathcal{P}_{q, n/r}\\r\ge1}}\frac{d(Q)}{|Q|^{\ell r/2}}\sum_{\substack{F\in \mathcal{H}_{q^2, D\left(\ell \right)/2}\\P|F \Rightarrow P \notin \mathbb{F}_q[t]}}\left[\chi_F\left( Q^{\ell r} \right)+\overline{\chi_F\left( Q^{\ell r} \right)}\right],
	\end{split}
\end{equation}
and the non-$\ell^{th}$ power contribution is
\begin{equation}
	\label{error term nK}
	\begin{split}
		&E^{\mathrm{nK}}_{\ell}(\Phi,D\left(\ell\right)-2)\\
		&=
		\frac{1}{(D\left(\ell\right)-2)\left|\mathcal{C}^{\mathrm{nK}}_\ell\left(g\right)\right|}\sum_{1\le n\le N}\hat{\Phi}\left(\frac{n}{D\left(\ell\right)-2}\right)\sum_{\substack{f\in \mathcal{M}_{q, n}\\f \ \text{non-$\ell^{th}$ power}}}\frac{\Lambda(f)}{|f|^{1/2}}\sum_{\substack{F\in \mathcal{H}_{q^2, D\left(\ell \right)/2}\\P|F \Rightarrow P \notin \mathbb{F}_q[t]}}\left[\chi_F(f)+\overline{\chi_F(f)}\right].
	\end{split}
\end{equation}\\


\section{The Kummer setting} \label{section 4}
We compute the one-level density of zeros of cubic and quartic Dirichlet $L$-functions in the Kummer setting. We note that the sextic case with the analogous affine model as the quartic case can be done with similar methods.\\
\subsection{The cubic case}
\subsubsection{The Main Term}

First we compute the following sum over the family of cubic Dirichlet $L$-functions with primitive cubic characters of genus $g$ and derive the size of the family $\left|\mathcal{C}_3^{\mathrm{K}}(g)\right|$.
\begin{lemma}
	\label{lemma 1 l=3 K}
	For $f$ a monic polynomial, let 
	\begin{equation*}
		\mathcal{T}_{1}=\sum_{F\in\mathcal{C}_3^{\mathrm{K}}(g)}1= \sum_{\substack{d_1+d_2=g+1\\d_1+2d_2\equiv 1\Mod 3}}
		\sum_{\substack{ F_1\in \mathcal{H}_{q,d_1} \\ (F_1,f)=1  }}\sum_{\substack{ F_2\in \mathcal{H}_{q,d_2} \\ (F_2,F_1f)=1  }}1.
	\end{equation*}	
	Then
	\begin{equation*}
		\begin{split}
			\mathcal{T}_1=\frac{q^{g+1}}{3}\left[ \left(g+2\right)J\left( \frac{1}{q},\frac{1}{q} \right) - \frac{\frac{d}{du_1}J(u_1,u_1)\mid_{1/q}}{q} 
			-
			2\mathfrak{Re}\left( \frac{J\left(\frac{1}{q},\frac{\zeta_3}{q}\right)\zeta_3^{1+2a}}{1-\zeta_3} \right) \right]+ O\left( q^{g/3+\epsilon g} \right),
		\end{split}
	\end{equation*}
	where 
	\begin{equation*}
		\begin{split}
			J(x,y)=\prod_{P\in \mathcal{P}_q}\left(1+x^{d(P)} +y^{d(P)}\right)\left(1-x^{d(P)} \right)\left(1-y^{d(P)} \right) \prod_{\substack{P\in \mathcal{P}_q\\ P\mid f}}\left(1+x^{d(P)}+y^{d(P)}\right)^{-1},
		\end{split}
	\end{equation*}
	and $\zeta_3=e^{2\pi i/3}$ is the third root of unity.
\end{lemma}
\begin{proof}
	We consider the generating series for the sums over $F_1$ and $F_2$ in $\mathcal{T}_{1}$, which is
	\begin{equation*}
		\mathcal{S}
		(u_1,u_2)=\sum_{\substack{F_1\in \mathcal{H}_q\\ (F_1, f)=1}}u_1^{d(F_1)}
		\sum_{\substack{F_2\in \mathcal{H}_q\\ (F_2, F_1f)=1}}u_2^{d(F_2)}.
	\end{equation*}
	Now, we can write this as the product over primes as
	\begin{equation*}
		\mathcal{S}
		(u_1, u_2)=\frac{\prod_{P\in \mathcal{P}_q}\left(1+\frac{u_1^{d(P)}}{1+u_2^{d(P)}}\right)\prod_{P\in \mathcal{P}_q}\left(1+u_2^{d(P)}\right)}{\prod_{\substack{P\in \mathcal{P}_q\\P\mid f}}\left(1+\frac{u_1^{d(P)}}{1+u_2^{d(P)}}\right)
			\prod_{\substack{P\in \mathcal{P}_q\\P \mid f}}\left(1+u_2^{d(P)}\right)} 
		=\frac{\prod_{P\in \mathcal{P}_q}\left(1+u_1^{d(P)}+u_2^{d(P)}\right)}{\prod_{\substack{P\in \mathcal{P}_q\\ P\mid f}}\left(1+u_1^{d(P)}+u_2^{d(P)}\right)} .
	\end{equation*}
	Hence 
	\begin{equation*}
		\mathcal{S}
		(u_1,u_2)=\mathcal{Z}_q(u_1)\mathcal{Z}_q(u_2)J(u_1,u_2),
	\end{equation*}
	where 
	\begin{equation*}
		\begin{split}
			J&(u_1,u_2)=\prod_{P\in \mathcal{P}_q}\left(1+u_1^{d(P)} +u_2^{d(P)}\right)\left(1-u_1^{d(P)} \right)\left(1-u_2^{d(P)} \right) \prod_{\substack{P\in \mathcal{P}_q\\ P\mid f}}\left(1+u_1^{d(P)}+u_2^{d(P)}\right)^{-1}\\
			=&\prod_{P\in \mathcal{P}}\left( 1-u_1^{2d(P)}-u_2^{2d(P)}-(u_1u_2)^{d(P)}+(u_1^2u_2)^{d(P)}+(u_1u_2^2)^{d(P)}\right)\prod_{\substack{P\in \mathcal{P}_q\\ P\mid f}}\left(1+u_1^{d(P)}+u_2^{d(P)}\right)^{-1}.
		\end{split}
	\end{equation*}
	Note that $J\left( u_1, u_2\right)$ has analytic continuation when $|u_1|<q^{-1/3}$ and $|u_2|<q^{-1/3}$.
	
	Using Perron's formula twice,
	\begin{equation*}
		\mathcal{T}_1
		=\sum_{\substack{d_1+d_2=g+1\\d_1+2d_2\equiv 1\Mod 3}}
		\frac{1}{(2\pi i)^2}\oint\oint \frac{J(u_1,u_2)}{(1-qu_1)(1-qu_2)u_1^{d_1}u_2^{d_2}}\frac{du_2}{u_2}\frac{du_1}{u_1}.
	\end{equation*}
	Since $g$ is fixed, we let $g\equiv b\pmod3$ for some $b\in\{0, 1, 2\}$ and take the difference of the equations 
	\begin{equation*}
		d_1+d_2=g+1, \ \text{and} \ d_1+2d_2\equiv1\Mod{3}
	\end{equation*}
	to obtain that 
	$$d_2\equiv b\Mod{3}.$$
	Thus for some integer $k$ and a fixed $a\in\{0,1,2\}$, 
	$$d_1=3k+a.$$
	Then, by analyzing the cases, the sum over $d_1$ is a sum over integers $0\le k\le [g/3]$, where $[x]$ denotes the closest integer to $x$.
	
	Computing the sum inside the integrals first,
	\begin{equation*}
		\begin{split}
			\mathcal{T}_1
			&=
			\sum_{k=0}^{[g/3]}
			\frac{1}{(2\pi i)^2}\oint\oint \frac{J(u_1,u_2)}{(1-qu_1)(1-qu_2)u_1^{3k+a}u_2^{g+1-3k-a}}\frac{du_2}{u_2}\frac{du_1}{u_1}\\
			&=
			\frac{1}{(2\pi i)^2}\oint_{|u_1|=q^{-3}}\oint_{|u_2|=q^{-2}} \frac{J(u_1,u_2)}{(1-qu_1)(1-qu_2)(u_2^3-u_1^3)}\left[ \frac{u_2^{2+a-b}}{u_1^{g+a-b}} -\frac{u_1^{3-a}}{u_2^{g+1-a}}\right]\frac{du_2}{u_2}\frac{du_1}{u_1}.
		\end{split}
	\end{equation*}  
	We write the integral above as the difference of two integrals. Note that the second one vanishes since the integrand over $u_1$ has no poles inside the circle $|u_1|=q^{-3}$.\\
	Hence
	\begin{equation*}
		{\mathcal{T}_{1}}=\frac{1}{(2\pi i)^2}\oint_{|u_1|=q^{-3}}\oint_{|u_2|=q^{-2}} \frac{u_2^{2+a-b}J(u_1,u_2)}{u_1^{g+a-b}(1-qu_1)(1-qu_2)(u_2^3-u_1^3)} \frac{du_2}{u_2}\frac{du_1}{u_1},
	\end{equation*} 
	where the integrand has poles $u_1, \zeta_3 u_1, \zeta_3^2u_1$ integrating over $u_2$. Here $\zeta_3$ denotes the third root of unity $e^{2\pi i/3}$.
	
	Computing the residue at the poles above, we have
	\begin{equation}
		\label{equation it}
		\begin{split}
			\mathcal{T}_1
			=
			\frac{1}{2\pi i}\oint_{|u_1|=q^{-3}}\frac{1}{3u_1^{g+1}(1-qu_1)}
			\left[ \frac{J(u_1,u_1)}{1-qu_1}+\frac{J(u_1,\zeta_3u_1)\zeta_3^{2+a-b}}{1-q\zeta_3u_1}+ \frac{J(u_1,\zeta_3^2u_1)\zeta_3^{2(2+a-b)}}{1-q\zeta_3^2u_1} \right]\frac{du_1}{u_1}.
		\end{split}
	\end{equation}
	Now to integrate over $u_1$, we write \eqref{equation it} as the sum of three integrals 
	\begin{equation}
		\label{sum T1}
		\mathcal{T}_1=T_1+T_2+T_3.
	\end{equation} 
	For each $T_i$ we shift the contour to $|u_1|=q^{-1/3+\epsilon}$, compute the residue at the corresponding poles at either $1/q$, $1/\zeta_3q$, or $1/\zeta_3^2q$ and bound the integral on the circle $|u_1|=q^{-1/3+\epsilon}$.
	
	$T_1$ has a double pole at $1/q$, and we obtain
	\begin{equation*}
		\begin{split}
			T_1=\frac{1}{2\pi i}\oint_{|u_1|=q^{-3}}\frac{J(u_1,u_1)}{3u_1^{g+1}(1-qu_1)^2}\frac{du_1}{u_1}
			=
			\frac{q^{g+1}}{3}\left[ (g+2)J\left( \frac{1}{q},\frac{1}{q} \right) - \frac{\frac{d}{du_1}J(u_1,u_1)\mid_{1/q}}{q} \right]+ O\left( q^{g/3+\epsilon g} \right).
		\end{split}
	\end{equation*}
	$T_2$ has two simple poles at $u_1=1/q$ and $u_1=1/\zeta_3q$. Hence we have 
	\begin{equation*}
		\begin{split}
			T_2=&\frac{1}{2\pi i}\oint_{|u_1|=q^{-3}}\frac{\zeta_3^{2+a-b}J(u_1,\zeta_3u_1)}{3u_1^{g+1}(1-qu_1)(1-\zeta_3qu_1)}\frac{du_1}{u_1}\\
			=&
			\frac{q^{g+1}}{3}\left[  \frac{J\left(\frac{1}{q},\frac{\zeta_3}{q}\right)\zeta_3^{1+2a}}{1-\zeta_3} + \frac{J\left(\frac{\zeta_3^2}{q},\frac{1}{q}\right)\zeta_3^{a}}{1-\zeta_3^2} \right]+ O\left( q^{g/3+\epsilon g} \right).
		\end{split}
	\end{equation*}
	$T_3$ has two simple poles at $u_1=1/q$ and $u_1=1/\zeta_3^2q$. We obtain
	\begin{equation*}
		\begin{split}
			T_3=&\frac{1}{2\pi i}\oint_{|u_1|=q^{-3}}\frac{\zeta_3^{2(2+a-b)}J\left(u_1,\zeta_3^2u_1\right)}{3u_1^{g+1}(1-qu_1)(1-\zeta_3^2qu_1)}\frac{du_1}{u_1}\\
			=&
			\frac{q^{g+1}}{3}\left[   \frac{J\left(\frac{1}{q}, \frac{\zeta_3^2}{q}\right)\zeta_3^{2+a}}{1-\zeta_3^2} + \frac{J\left(\frac{\zeta_3}{q}, \frac{1}{q}\right)\zeta_3^{2a}}{1-\zeta_3}\right]+ O\left( q^{g/3+\epsilon g} \right)\\
			=&\zeta_3^2T_2.
		\end{split}
	\end{equation*}
	where the last equality holds since $J(u_1,u_2)=J(u_2,u_1)$. 
	
	We can simplify the sum in \eqref{sum T1}.
	$$\mathcal{T}_1=T_1+T_2+\zeta_3^2T_2=T_1-\zeta_3T_2.$$
	Thus
	\begin{equation*}
		\begin{split}
			\mathcal{T}_1
			=
			\frac{q^{g+1}}{3}\left[ (g+2)J\left( \frac{1}{q},\frac{1}{q} \right) - \frac{\frac{d}{du_1}J(u_1,u_1)\mid_{1/q}}{q} 
			-
			\frac{J\left(\frac{1}{q},\frac{\zeta_3}{q}\right)\zeta_3^{1+g}}{1-\zeta_3}
			-
			\frac{J\left(\frac{\zeta_3^2}{q},\frac{1}{q}\right)\zeta_3^{2+2g}}{1-\zeta_3^2}  
			 \right]
			 + O\left( q^{g/3+\epsilon g} \right),
		\end{split}
	\end{equation*}
	where we use the fact $2g+1\equiv a\pmod3$.
\end{proof}

The following corollary follows easily from Lemma \ref{lemma 1 l=3 K}. The size fo the family is the following.
\begin{corollary}
	Let $\displaystyle J_0\left(u_1,u_2 \right)=\prod_{P\in \mathcal{P}}\left( 1+u_1^{d(P)}+u_2^{d(P)}\right)\left(1-u_1^{d(P)}  \right)\left(1-u_2^{d(P)}  \right),$ then
	\begin{equation*}
		\begin{split}
			\left|\mathcal{C}^{\mathrm{K}}_{3}(g)\right|
			=
			\frac{q^{g+1}}{3}\left[ \left(g+2\right)J_0\left( \frac{1}{q},\frac{1}{q} \right) - \frac{\frac{d}{du_1}J_0(u_1,u_1)\mid_{1/q}}{q}
			-\frac{J_0\left(\frac{1}{q},\frac{\zeta_3}{q}\right)\zeta_3^{1+g}}{1-\zeta_3}
			-
			\frac{J_0\left(\frac{\zeta_3^2}{q},\frac{1}{q}\right)\zeta_3^{2+2g}}{1-\zeta_3^2}  
			\right]\\
			+ O\left( q^{g/3+\epsilon g} \right).
		\end{split}
	\end{equation*}
\end{corollary}

Now we compute the main term $M_3^{\mathrm{K}}\left( \Phi, D(3)-2 \right)$ of the one-level density of zeros of cubic Dirichlet $L$-functions in the Kummer setting. We use notations in Section \ref{notations} as needed.
\begin{lemma}
	\label{lemma 2 l=3 main term}
	Recall the main term of the one-level density of zeros in the Kummer setting is given in \eqref{main term K}. Let $\ell=3$ for cubic Dirichlet $L$-functions with characters of genus $g$, we have
		\begin{equation*}
			\begin{split}
				M^{\mathrm{K}}_{3}\left(\Phi, g \right)
				=&
				\frac{2}{g}\sum_{1\le n\le N/3}\hat{\Phi}\left(\frac{3n}{g}\right)\sum_{\substack{Q\in \mathcal{P}_{q, n/r}\\r\ge1}}\frac{d(Q)}{|Q|^{3r/2}\left(1+2|Q|^{-1}\right)}\\
				+& 4\mathfrak{Re}\left(\frac{h_1}{g}\sum_{1\le n\le N/3}\hat{\Phi}\left(\frac{3n}{g}\right)\sum_{\substack{Q\in \mathcal{P}_{q, n/r}\\r\ge1}}\frac{d(Q)\left(1-c_Q\right)}{|Q|^{3r/2}\left(1+2|Q|^{-1}\right)}\right)\\
				+&
				\frac{2{h_2}}{g}\sum_{1\le n\le N/3}\hat{\Phi}\left(\frac{3n}{g}\right)\sum_{\substack{Q\in \mathcal{P}_{q, n/r}\\r\ge1}}\frac{2d(Q)^2|Q|^{-1}}{|Q|^{3r/2}\left(1+2|Q|^{-1}\right)^2}
				+O\left( 
				q^{-N/3}q^{-2g/3}q^{\epsilon g} \right),
			\end{split}
		\end{equation*} 
	where $c_Q, h_1, h_2$ are explicitly defined in \eqref{constant c_Q}, \eqref{constant h_1} and \eqref{constant h_2} respectively.
\end{lemma}
\begin{proof}
	Recall that for $\ell =3$, $D(3)-2=g$ given by \eqref{equation D(l)}, and the main term of the one-level density of zeros described in \eqref{main term K} comes from when $f$ is a cube. For some monic irreducible polynomial $Q$ and some positive integer $r$,
	$$f=Q^{3r}.$$ 
	Thus $\chi_F(f)=\overline{\chi_F(f)}=1$, and we can write  
	$$M^{\mathrm{K}}_{3}\left(\Phi, g \right)
	=
	\frac{2}{g|\mathcal{C}_3^{\mathrm{K}}\left(g\right)|}\sum_{1\le n\le N/3}\hat{\Phi}\left(\frac{3n}{g}\right)\sum_{\substack{Q\in \mathcal{P}_{q, n/r}\\r\ge1}}\frac{d(Q)\mathcal{T}_1}{|Q|^{3r/2}}.$$
	
	Let $$D_f(u_1,u_2)=\prod_{\substack{P\in \mathcal{P}_q\\P\mid f}}\left( 1+u_1^{d(P)}+u_2^{d(P)} \right)^{-1}.$$
	Since $f=Q^{3r}$, we have
	$$D_f(u_1,u_1)=D_Q(u_1,u_1)=\left( 1+u_1^{d(Q)}+u_2^{d(Q)} \right)^{-1},$$
	and
	$$J(u_1,u_2)=J_0(u_1,u_2)D_Q\left(u_1,u_2\right).$$
	
	Now, using Lemma \ref{lemma 1 l=3 K} for $\mathcal{T}_1$, we can write the sum over $Q$ as two sums 
	$$H_1+H_2+O\left( q^{-n}q^{g/3+\epsilon g} \right),$$ 
	grouped by the same order of derivation on $D_Q(u_1,u_2)$. \\
	We have
	\begin{equation*}
		\begin{split}
			H_1
			=&\frac{q^{g+1}}{3}\sum_{\substack{Q\in \mathcal{P}_{q, n/r}\\r\ge1}}\frac{d(Q)}{|Q|^{3r/2}\left(1+2|Q|^{-1}\right)}
			\times\\
			&\left[ (g+2)J_0\left( \frac{1}{q},\frac{1}{q} \right) - \frac{\frac{d}{du_1}J_0(u_1,u_1)\mid_{1/q}}{q} 
			-
			\frac{c_QJ_0\left(\frac{1}{q},\frac{\zeta_3}{q}\right)\zeta_3^{1+g}}{1-\zeta_3}
			-
			\frac{\overline{c_Q}J_0\left(\frac{\zeta_3^2}{q},\frac{1}{q}\right)\zeta_3^{2+2g}}{1-\zeta_3^2}  
			\right],
		\end{split}
	\end{equation*}
	where we denote
	\begin{equation}
		\label{constant c_Q}
		c_Q=\frac{D_Q\left( \frac{1}{q},\frac{\zeta_3}{q} \right)}{D_Q\left( \frac{1}{q},\frac{1}{q}\right)}.
	\end{equation}
	Taking the derivative of $D_Q\left(u_1,u_1  \right)$, we obtain
	\begin{equation*}
		\begin{split}
			H_2=
			-\frac{q^{g+1}J_0\left( \frac{1}{q},\frac{1}{q} \right)}{3}   
			\sum_{\substack{Q\in \mathcal{P}_{q, n/r}\\r\ge1}}\frac{2d(Q)^2|Q|^{-1}}{|Q|^{3r/2}\left(1+2|Q|^{-1}\right)^2}.
		\end{split}
	\end{equation*}

	We now consider $H_i/\left| \mathcal{C}^{\mathrm{K}}_{3}(g) \right|$.
	\begin{equation*}
		\begin{split}
			\frac{H_1}{\left| \mathcal{C}^{\mathrm{K}}_{3}(g) \right|}
			=&
			\sum_{\substack{Q\in \mathcal{P}_{q, n/r}\\r\ge1}}\frac{d(Q)}{|Q|^{3r/2}\left(1+2|Q|^{-1}\right)}
			+
			2\mathfrak{Re}\left(h_1\sum_{\substack{Q\in \mathcal{P}_{q, n/r}\\r\ge1}}\frac{d(Q)\left(1-c_Q\right)}{|Q|^{3r/2}\left(1+2|Q|^{-1}\right)}\right),
		\end{split}
	\end{equation*}
	where $h_1$ denotes the constant
	\begin{equation}
		\label{constant h_1}
		h_1=
		\frac{q^{g+1}J_0\left(\frac{1}{q},\frac{\zeta_3}{q}\right)\zeta_3^{1+g}}{3 \left| \mathcal{C}^{\mathrm{K}}_{3}(g) \right| \left(1-\zeta_3\right)}.
	\end{equation}
	We remark that $h_1$ is of order $1/g$, and when $d(Q)\equiv 0\pmod3$, $c_Q=1$ and latter two terms in the sum vanish.
	
	Let
	\begin{equation}
		\label{constant h_2}
		h_2=-\frac{q^{g+1}J_0\left( \frac{1}{q},\frac{1}{q} \right)}{3\left| \mathcal{C}^{\mathrm{K}}_{3}(g) \right|}.
	\end{equation}
	Thus we have
	\begin{equation*}
		\begin{split}
			M^{\mathrm{K}}_{3}\left(\Phi, g \right)
			=&
			\frac{2}{g}\sum_{1\le n\le N/3}\hat{\Phi}\left(\frac{3n}{g}\right)\sum_{\substack{Q\in \mathcal{P}_{q, n/r}\\r\ge1}}\frac{d(Q)}{|Q|^{3r/2}\left(1+2|Q|^{-1}\right)}\\
			+&4\mathfrak{Re}\left( \frac{h_1}{g}\sum_{1\le n\le N/3}\hat{\Phi}\left(\frac{3n}{g}\right)\sum_{\substack{Q\in \mathcal{P}_{q, n/r}\\r\ge1}}\frac{d(Q)\left(1-c_Q\right)}{|Q|^{3r/2}\left(1+2|Q|^{-1}\right)}\right)\\
			+&
			\frac{2{h_2}}{g}\sum_{1\le n\le N/3}\hat{\Phi}\left(\frac{3n}{g}\right)\sum_{\substack{Q\in \mathcal{P}_{q, n/r}\\r\ge1}}\frac{2d(Q)^2|Q|^{-1}}{|Q|^{3r/2}\left(1+2|Q|^{-1}\right)^2}
			+O\left( 
			q^{-N/3}q^{-2g/3}q^{\epsilon g} \right),
		\end{split}
	\end{equation*} 
	where $h_2$ is also of order $1/g$.
\end{proof}

%
\subsubsection{The Error Term}
Recall that if $f$ is not a cube, we have the error term $E^{\mathrm{K}}_{3}\left( \Phi, g \right)$ of the one-level density of zeros expressed in \eqref{error term K}. 
We prove the following upper bound.
\begin{lemma}
	\label{lemma 1 l=3 error term}
	Let $f\in\mathbb{F}_q[t]$ be a monic polynomial and $\chi_{3}$ as defined in \eqref{equation chi_ell}. Let
	$$\mathcal{T}_{2}
	=
	\sum_{ \substack{\chi \ \text{primitive cubic }\\ \text{genus}(\chi)=g\\ \chi\mid_{\mathbb{F}_q^\times}=\chi_3} }\chi(f).$$
	Then 
	\begin{equation*}
		\begin{split}
			\mathcal{T}_{2}
			=
			\sum_{ \substack{d_1+d_2=g+1 \\ d_1+2d_2\equiv 1 \Mod{3} }} \sum_{F_1\in \mathcal{H}_{q,d_1}}\chi_f(F_1) \sum_{ \substack{F_2\in \mathcal{H}_{q,d_2} \\ (F_2,F_1)=1}}\chi_f(F_2)^2
			\ll
			gq^{g/2}q^{\epsilon d(f)}.
		\end{split}
	\end{equation*}
	The error term for the one-level density of zeros of Dirichlet $L$-functions in the Kummer setting is given in \eqref{error term K}. For $\ell=3$, we have
	$$E^{\mathrm{K}}_{3}\left( \Phi, g \right) 
	\ll
	q^{N/2}q^{-g/2}q^{\epsilon N}.
	$$  
\end{lemma}
\begin{proof}
	First we give an upper bound on $\mathcal{T}_2$.\\
	By Lemma \ref{lemma K main term}, we have
	$$
	\mathcal{T}_2=
	\sum_{ \substack{d_1+d_2=g+1 \\ d_1+2d_2\equiv 1 \Mod{3} }} \sum_{F_1\in \mathcal{H}_{q,d_1}}\chi_f(F_1) \sum_{ \substack{F_2\in \mathcal{H}_{q,d_2} \\ (F_2,F_1)=1}}\chi_f(F_2)^2.$$
	Then we consider the generating series for the sums over $F_1$ and $F_2$,
	$$\mathcal{S}
	(u_1,u_2)=\sum_{F_1\in \mathcal{H}_{q}}\chi_f(F_1)u_1^{d(F_1)} \sum_{ \substack{F_2\in \mathcal{H}_{q} \\ (F_2,F_1)=1}}\chi_f(F_2)^2u_2^{d(F_2)}.$$
	
	The sum over $F_2$ can be written as the product
	$$\frac{ \prod_{\substack{P\in \mathcal{P}_q}} \left( 1+\chi_f(P)^2u_2^{d(P)} \right)} {\prod_{P \mid F_1}\left( 1+\chi_f(P)^2u_2^{d(P)} \right)},$$
	thus factoring out the $L$-functions
	$$\mathcal{S}
	(u_1,u_2)=
	\frac{\mathcal{L} \left(u_2,\chi_f^2\right)}{\mathcal{L}\left(u_2^2,\chi_f \right)}\sum_{F_1\in \mathcal{H}_{q}}\frac{\chi_f(F_1)u_1^{d(F_1)}}{\prod_{P \mid F_1}\left( 1+\chi_f(P)^2u_2^{d(P)} \right)} .$$
	Similarly, we write the sum over $F_1$ as the product
	$$\prod_{\substack{P\in \mathcal{P}_q }}\left( 1+\frac{\chi_f(P)u_1^{d(P)}}{1+\chi_f(P)^2u_2^{d(P)}} \right)
	=\frac{\mathcal{L} \left(u_1,\chi_f\right)}{\mathcal{L}\left(u_1^2, \chi_f^2\right)}\prod_{P \in \mathcal{P}_q}\frac{1+\chi_f(P)u_1^{d(P)}+\chi_f(P)^2u_2^{d(P)}}{\left( 1+\chi_f(P)u_1^{d(P)} \right)\left(1+\chi_f(P)^2u_2^{d(P)}  \right)}.$$
	Combining the two, we obtain
	\begin{equation*}
		\begin{split}
			\mathcal{S}
			(u_1,u_2)
			=\frac{\mathcal{L} \left(u_1,\chi_f\right)\mathcal{L} \left(u_2,\chi_f^2\right)} {\mathcal{L}\left(u_1^2, \chi_f^2\right)\mathcal{L}\left(u_2^2, \chi_f\right)} 
			\prod_{P \in \mathcal{P}_q}\frac{1+\chi_f(P)u_1^{d(P)}+\chi_f(P)^2u_2^{d(P)}}{\left( 1+\chi_f(P)u_1^{d(P)}\right)\left(1+\chi_f(P)^2u_2^{d(P)} \right)}.
		\end{split}
	\end{equation*}
	Note here that the product over $P$ is absolutely convergent for $|u_1|,|u_2|<q^{-1/2}.$
	
	Using Perron's formula,
	\begin{equation*}
		\begin{split}
			\mathcal{T}_2
			=
			\sum_{ \substack{d_1+d_2=g+1 \\ d_1+2d_2\equiv 1 \Mod{3} }}\frac{1}{\left(2\pi i\right)^2}\oint_{|u_1|=q^{-1/2}} \oint_{|u_2|=q^{-1/2}}  \frac{\mathcal{S}(u_1,u_2)}{u_1^{d_1}u_2^{d_2}}\frac{du_1}{u_1}\frac{du_2}{u_2}.
		\end{split}
	\end{equation*}
	Then by the Lindel\"{o}f hypothesis (Lemma \ref{Lindelof upper bound} and Lemma \ref{Lindelof lower bound}), we obtain the following bounds on the integrals. For $i\in \{1,2\}$
	\begin{equation*}
		\frac{1}{2\pi i}\oint_{|u_i|=q^{-1/2}} \frac{\mathcal{L}\left( u_i, \chi_f^i\right)}{\mathcal{L}\left( u_i^2, \chi_f^{2i}\right)u_i^{d_i}}\frac{d}{du_i}
		\ll
		q^{d_i/2}q^{\epsilon d(f)}.
	\end{equation*}
	Hence 
	$$\mathcal{T}_2
	\ll
	\sum_{ \substack{d_1=1 }}^{g+1}q^{g/2}q^{\epsilon d(f)}
	\ll
	q^{g/2}q^{\epsilon d(f)}.
	$$
	
	Now for $\overline{\mathcal{T}_2}$, since Lemma \ref{Lindelof upper bound} and Lemma \ref{Lindelof lower bound} hold for $\mathcal{L}\left(u_i, \overline{\chi_f}^i\right)$ and $\mathcal{L}\left(u_i^2, \overline{\chi_f}^{2i}\right)$ for $i\in \{1,2\}$, we have
	$$\overline{\mathcal{T}_2}
	\ll
	q^{g/2}q^{\epsilon d(f)}.
	$$
	
	The error term in the Kummer setting is as given in \eqref{error term K} for $\ell=3$.
	\begin{equation*}
		E^{\mathrm{K}}_{3}\left( \Phi, g \right)
		=
		\frac{1}{g\left|\mathcal{C}_3^\mathrm{K}(g) \right|} \sum_{1\le n\le N}\hat{\Phi}\left(\frac{n}{g}\right)\sum_{\substack{f\in \mathcal{M}_{q,n}\\ f \ \text{non} \text{cube}}}\frac{\Lambda(f)}{|f|^{1/2}}\left(\mathcal{T}_2+\overline{\mathcal{T}_2}\right).
	\end{equation*}
	We trivially bound the double sum over $f$ and $n$, then divide by $\left| \mathcal{C}^{\mathrm{K}}_{3}(g) \right|$ to obtain
	\begin{equation*}
		E^{\mathrm{K}}_{3}\left( \Phi, g \right)
		\ll
		q^{N/2}q^{-g/2}q^{\epsilon N}.
	\end{equation*}
\end{proof}

\subsubsection{Proofs for $\ell=3$ Kummer setting results}
\begin{proof}[Proof of Theorem \ref{Theorem 1}]
	We use some notations in Section \ref{notations} and recall that $\mathcal{C}_3^\mathrm{K}\left(g\right)$ denotes the family of cubic Dirichlet $L$-functions in the Kummer setting \ref{notation K fam}. The one-level density of zeros for cubic Dirichlet $L$-functions in the Kummer setting is 
	\begin{equation}\label{equation 1LD K3}
		\Sigma_3^\mathrm{K}\left(\Phi, g\right)
		=
		\hat{\Phi}(0)
		-
		\mathcal{A}_{3}^{\mathrm{K}}\left(\Phi, g\right),
	\end{equation}
	where 
	\begin{equation*}
		\mathcal{A}^{\mathrm{K}}_{3}\left(\Phi, g \right)
		=\frac{1}{g|\mathcal{C}^{\mathrm{K}}_3(g)|}\sum_{1\le n\le N}\hat{\Phi}\left(\frac{n}{g}\right)\sum_{f\in \mathcal{M}_{q,n}}\frac{\Lambda(f)}{|f|^{1/2}}{\sum_{\substack{F\in\mathcal{C}_3^\mathrm{K}\left(g\right)\\ }}}\left[\chi_F(f) + \overline{\chi_F(f)} \right].
	\end{equation*}
	Note that $\Sigma_3^\mathrm{K}\left(\Phi, g\right)$ is the $\ell=3$ case of \eqref{1LD K} and $\mathcal{A}^{\mathrm{K}}_{3}\left(\Phi, g \right)$ is the $\ell=3$ case in \eqref{equation A Kummer}.
	
	Using Lemma \ref{lemma 2 l=3 main term} and Lemma \ref{lemma 1 l=3 error term}, we have the following result given in
	\eqref{equation final l=3}.
	\begin{equation*}
		\begin{split}
			\Sigma_3^\mathrm{K}\left(\Phi, g\right)
			=
			\hat{\Phi}&(0)
			-
			\frac{2}{g}\sum_{1\le n\le N/3}\hat{\Phi}\left(\frac{3n}{g}\right)\sum_{\substack{Q\in \mathcal{P}_{q, n/r}\\r\ge1}}\frac{d(Q)}{|Q|^{3r/2}\left(1+2|Q|^{-1}\right)}\\
			+&4\mathfrak{Re}\left( \frac{h_1}{g}\sum_{1\le n\le N/3}\hat{\Phi}\left(\frac{3n}{g}\right)\sum_{\substack{Q\in \mathcal{P}_{q, n/r}\\r\ge1}}\frac{d(Q)\left(1-c_Q\right)}{|Q|^{3r/2}\left(1+2|Q|^{-1}\right)}\right)\\
			+&
			\frac{2{h_2}}{g}\sum_{1\le n\le N/3}\hat{\Phi}\left(\frac{3n}{g}\right)\sum_{\substack{Q\in \mathcal{P}_{q, n/r}\\r\ge1}}\frac{2d(Q)^2|Q|^{-1}}{|Q|^{3r/2}\left(1+2|Q|^{-1}\right)^2}
			+O\left(q^{N/2}q^{-g/2}q^{\epsilon N}\right).
		\end{split}
	\end{equation*}
	Here $c_Q$, $h_1$ and $h_2$ are explicitly defined in the main term lemma, Lemma \ref{lemma 2 l=3 main term}, by \eqref{constant c_Q}, \eqref{constant h_1} and \eqref{constant h_2} respectively.
\end{proof}
Computing the limit as $g\rightarrow \infty$, we confirm the symmetry type of the family.
\begin{proof} [Proof of corollary 2.4 for $\ell=3$ in the Kummer setting]
	Let $N<g$. Then
	$$\lim_{g\rightarrow \infty}\Sigma_3^\mathrm{K}\left(\Phi, g\right)
	=\hat{\Phi}(0),
	$$	
	since the double sums over $n$ and $Q$ above are of constant size.
	
	Furthermore, we compute the two integrals below and confirm that
	$$
	\hat{\Phi}(0)
	=
	\int_{-\infty}^{\infty}\hat{\Phi}(y)\hat{\mathcal{W}}_{U(g)}(y)dy
	=
	\int_{-\infty}^{\infty}\hat{\Phi}(y)\delta_0(y)dy,$$
	where $\mathcal{W}_{U(g)}(y)=\delta_0(y)$ denotes the one-level scaling density of the group of $g\times g$ unitary matrices.
\end{proof}
This proves that the symmetry type of the family is unitary and it supports Katz and Sarnak's philosophy.
\\
\\


\subsection{The quartic case}
\subsubsection{The Main Term}

First we compute the following sum over the family of quartic Dirichlet $L$-functions and derive the size of the family $\left|\mathcal{C}_4^{\mathrm{K}}(g)\right|$.
\begin{lemma}
	\label{lemma 1 l=4 K}
	For $f$ a monic polynomial, let 
	\begin{equation*}
		\mathcal{K}_{1}=\sum_{\substack{d_1+d_3=\frac{2g}{3}+1\\d_1+3d_3\equiv 1\Mod 4}}
		\sum_{\substack{ F_1\in \mathcal{H}_{q,d_1} \\ (F_1,f)=1  }}\sum_{\substack{ F_3\in \mathcal{H}_{q,d_3} \\ (F_3,F_1f)=1  }}1.
	\end{equation*}	
	Then
	\begin{equation*}
		\begin{split}
			\mathcal{K}_1
			=
			\frac{q^{G+2}}{2}\left[ \left(G+1\right)J\left( \frac{1}{q},\frac{1}{q} \right) - \frac{\frac{d}{du_1}J(u_1,u_1)\mid_{1/q}}{q} \right]+ O\left( q^{G/3+\epsilon g} \right),
		\end{split}
	\end{equation*}
	where $\displaystyle G=\frac{2g}{3}+1$, and
	\begin{equation*}
		\begin{split}
			J(x,y)=\prod_{P\in \mathcal{P}_q}\left(1+x^{d(P)} +y^{d(P)}\right)\left(1-x^{d(P)} \right)\left(1-y^{d(P)} \right) \prod_{\substack{P\in \mathcal{P}_q\\ P\mid f}}\left(1+x^{d(P)}+y^{d(P)}\right)^{-1}.\\
		\end{split}
	\end{equation*}
\end{lemma}
\begin{proof}
	We first rewrite the outer most sum of $\mathcal{K}_1$.\\
	The degree of the conductor is the integer $\frac{2g}{3}+1$ given in \eqref{equation D(l)}, so $g\equiv 0\Mod{3}$. We let $g$ be even for convenience, so $d_3$ must be even. The case when $g$ is odd is the symmetric case when $d_1$ must be even. To simplify some notations, we let 
	\begin{equation}
		\label{Gv notations}
		v=\frac{g}{3} \ \ \text{and}\ \  G=\frac{2g}{3}+1=2v+1.
	\end{equation}
	
	Using $d_1+d_3=G$ and $d_1+3d_3\equiv1\Mod{4}$, we found that 
	\begin{equation*}
		2d_3\equiv 0 \Mod4 \implies d_3\equiv 0 \Mod{2}.
	\end{equation*}
	Since $G$ is odd, for some integer $k_1$,
	\begin{equation}
		\label{equation d_1}
		2k_1+1=d_1\le G.
	\end{equation}
	Lastly, $d_3=G-d_1$, which simplifies to 
	\begin{equation}
		\label{equation d_3}
		d_3=2\left(v-k_1\right).
	\end{equation}
	We also note that the congruence 
	$$d_1+3d_3\equiv 1 \Mod4$$
	is satisfied for all $d_1, d_3$ in agreement with equations \eqref{equation d_1} and \eqref{equation d_3}. Thus, we are summing over all non-negative integers $k_1\le v$.
	Hence we rewrite the sum as 
	\begin{equation*}
		\begin{split}
			\mathcal{K}_{1}
			=
			\sum_{\substack{  d_1=k_1\le v \\d_3=2\left(v-k_1\right) }} \sum_{\substack{ F_1\in \mathcal{H}_{q,d_1} \\ (F_1,f)=1  }}\sum_{\substack{ F_3\in \mathcal{H}_{q,d_3} \\ (F_3,F_1f)=1  }}1.
		\end{split}
	\end{equation*}
	
	Then, similar to the cubic case, we consider the generating series for the sum over $F_1$ and $F_3$ of $\mathcal{K}_{1}$
	\begin{equation*}
		\mathcal{S}
		(u_1,u_3)=\sum_{\substack{F_1\in \mathcal{H}_q\\ (F_1, f)=1}}u_1^{d(F_1)}
		\sum_{\substack{F_3\in \mathcal{H}_q\\ (F_3, F_1f)=1}}u_3^{d(F_3)}.
	\end{equation*}
	We can write it as the product, 
	\begin{equation*}
		\mathcal{S}
		(u_1, u_3)=\frac{\prod_{P\in \mathcal{P}_q}\left(1+\frac{u_1^{d(P)}}{1+u_3^{d(P)}}\right)\prod_{P\in \mathcal{P}_q}\left(1+u_3^{d(P)}\right)}{\prod_{\substack{P\in \mathcal{P}_q\\P\mid f}}\left(1+\frac{u_1^{d(P)}}{1+u_3^{d(P)}}\right)
			\prod_{\substack{P\in \mathcal{P}_q\\P \mid f}}\left(1+u_3^{d(P)}\right)} 
		=\frac{\prod_{P\in \mathcal{P}_q}\left(1+u_1^{d(P)}+u_3^{d(P)}\right)}{\prod_{\substack{P\in \mathcal{P}_q\\ P\mid f}}\left(1+u_1^{d(P)}+u_3^{d(P)}\right)} .
	\end{equation*}
	Hence 
	\begin{equation*}
		\mathcal{S}
		(u_1,u_3)=\mathcal{Z}_q(u_1)\mathcal{Z}_q(u_3)J(u_1,u_3),
	\end{equation*}
	where 
	\begin{equation*}
		\begin{split}
			J&(u_1,u_3)=\prod_{P\in \mathcal{P}_q}\left(1+u_1^{d(P)} +u_3^{d(P)}\right)\left(1-u_1^{d(P)} \right)\left(1-u_3^{d(P)} \right) \prod_{\substack{P\in \mathcal{P}_q\\ P\mid f}}\left(1+u_1^{d(P)}+u_3^{d(P)}\right)^{-1}\\
			=&\prod_{P\in \mathcal{P}}\left( 1-u_1^{2d(P)}-u_3^{2d(P)}-(u_1u_3)^{d(P)}+(u_1^2u_3)^{d(P)}+(u_1u_3^2)^{d(P)}\right)\prod_{\substack{P\in \mathcal{P}_q\\ P\mid f}}\left(1+u_1^{d(P)}+u_3^{d(P)}\right)^{-1}.
		\end{split}
	\end{equation*}
	Note that $J\left( u_1, u_3\right)$ has analytic continuation when $|u_1|<q^{-1/3}$ and $|u_3|<q^{-1/3}$.
	
	Using Perron's formula twice,
	\begin{equation*}
		\mathcal{K}_1
		=\sum_{\substack{k_1=0}}^v
		\frac{1}{(2\pi i)^2}\oint\oint \frac{J(u_1,u_3)}{(1-qu_1)(1-qu_3)u_1^{2k_1+1}u_3^{2v-2k_1}}\frac{du_3}{u_3}\frac{du_1}{u_1}.
	\end{equation*}
	Computing the sum over $k_1$ first, we have
	\begin{equation*}
		\begin{split}
			\sum_{k_1=0}^{v}\frac{1}{u_1^{2k_1+1}u_3^{2v-2k_1}}
			=&\frac{u_1}{u_3^{2v}\left(u_3^2-u_1^2\right)}\left[\left(\frac{u_3^2}{u_1^2}\right)^{v+1}-1\right]
			\\=&
			\frac{1}{\left( u_3^2-u_1^2 \right)}\left[\frac{u_3^2}{u_1^{2v+1}}-\frac{u_1}{u_3^{2v}}\right].
		\end{split}
	\end{equation*}
	Thus
	\begin{equation*}
		\begin{split}
			\mathcal{K}_1
			=
			\frac{1}{(2\pi i)^2}\oint_{|u_1|=q^{-3}}\oint_{|u_2|=q^{-2}} \frac{J(u_1,u_3)}{(1-qu_1)(1-qu_3)(u_3^2-u_1^2)}\left[ \frac{u_3}{u_1^{2v+2}}-\frac{1}{u_3^{2v+1}} \right]du_3du_1.
		\end{split}
	\end{equation*}  
	We write the integral above as the difference of two integrals correspondingly and note that the second one 
	\begin{equation*}
		\frac{1}{(2\pi i)^2}\oint_{|u_1|=q^{-3}}\oint_{|u_2|=q^{-2}} \frac{-J(u_1,u_3)}{u_3^{2v+1}(1-qu_1)(1-qu_3)(u_3^2-u_1^2)} du_3du_1=0
	\end{equation*}
	since the integrand over $u_1$ has no poles inside the circle $|u_1|=q^{-3}$.\\
	Hence
	\begin{equation*}
		{\mathcal{K}_{1}}=\frac{1}{(2\pi i)^2}\oint_{|u_1|=q^{-3}}\oint_{|u_3|=q^{-2}} \frac{u_3J(u_1,u_3)}{u_1^{2v+2}(1-qu_1)(1-qu_3)(u_3^2-u_1^2)} du_3du_1,
	\end{equation*} 
	where the poles of the integrand integrating over $u_3$ are $u_1, -u_1$. 
	
	Computing the residue at the poles above, we have
	\begin{equation}
		\label{equation du1}
		\begin{split}
			\mathcal{K}_1
			=
			\frac{1}{2\pi i}\oint_{|u_1|=q^{-3}}\frac{1}{2u_1^{2v+2}(1-qu_1)}
			\left[ \frac{J(u_1,u_1)}{1-qu_1}+\frac{J(u_1, -u_1)}{1+qu_1}\right]du_1.
		\end{split}
	\end{equation}
	Now to integrate over $u_1$, we write \eqref{equation du1} as the sum of two integrals 
	$$\mathcal{K}_1=K_1+K_{-1}$$ 
	where for $\beta \in\{1, -1\}$,
	\begin{equation*}
		K_{\beta}
		=
		\frac{1}{2\pi i}\oint_{|u_1|=q^{-3}}\frac{J(u_1, \beta u_1)}{2u_1^{2v+2}(1-qu_1)(1-q\beta u_1)}du_1.
	\end{equation*}
	For each $K_\beta$, we shift the contour to $|u_1|=q^{-1/3+\epsilon}$ and encounter the pole $1/q$ and $1/q\beta$. We compute residues at the corresponding poles and bound the integral on circle $|u_1|=q^{-1/3+\epsilon}$.
	
	$K_1$ has a double pole at $1/q$, and we obtain
	\begin{equation*}
		\begin{split}
			K_1
			=
			\frac{q^{G+2}}{2}\left[ \left(G+1\right)J\left( \frac{1}{q},\frac{1}{q} \right) - \frac{\frac{d}{du_1}J(u_1,u_1)\mid_{1/q}}{q} \right]+ O\left( q^{G/3+\epsilon g} \right).
		\end{split}
	\end{equation*}
	$K_{-1}$ has two simple poles at $u_1=1/q$ and $u_1=-1/q$. 
	\\
	Computing the residues, we have 
	\begin{equation*}
		\begin{split}
			K_{-1}
			=
			\frac{q^{G}}{4}\left[  J\left(\frac{1}{q},\frac{-1}{q}\right) - J\left(\frac{-1}{q},\frac{1}{q}\right) \right]+ O\left( q^{G/3+\epsilon g} \right),
		\end{split}
	\end{equation*}
	where since $\displaystyle J\left(\frac{1}{q},\frac{-1}{q}\right)= J\left(\frac{-1}{q},\frac{1}{q}\right)$, $K_{-1}=O\left( q^{G/3+\epsilon g} \right)$.
	
	Thus
	\begin{equation*}
		\begin{split}
			\mathcal{K}_1
			=
			\frac{q^{G+2}}{2}\left[ \left(G+1\right)J\left( \frac{1}{q},\frac{1}{q} \right) - \frac{\frac{d}{du_1}J(u_1,u_1)\mid_{1/q}}{q} \right]+ O\left( q^{G/3+\epsilon g} \right).
		\end{split}
	\end{equation*}
This gives the result as desired.
\end{proof}

From Lemma \ref{lemma 1 l=4 K}, the size of the family is the following.
\begin{corollary}
	Let $\displaystyle J_0\left(u \right)=\prod_{P\in \mathcal{P}}\left( 1 + 2u^{d(P)}\right)\left(1-u^{d(P)}  \right)^2,$ then
	\begin{equation*}
		\begin{split}
			\left|\mathcal{C}^{\mathrm{K}}_{4}(g)\right|
			=
			\frac{q^{G+2}}{2}\left[ \left(G+1\right)J_0\left( \frac{1}{q} \right) - \frac{\frac{d}{du}J_0(u)\mid_{1/q}}{q} \right]+ O\left( q^{G/3+\epsilon g} \right),
		\end{split}
	\end{equation*}
where $G=2g/3+1$ as in \eqref{Gv notations}.
\end{corollary}
Now we compute the main term of the one-level density of quartic Dirichlet $L$-functions over function fields. We use some notations as needed given in Section \ref{notations}. Furthermore, by Equation \eqref{equation D(l)}, we have $\displaystyle D(4)-2=\frac{2g}{3}$.
\begin{lemma}
	\label{lemma 2 l=4}
	Recall the main term of the one-level density of zeros in the Kummer setting is given in \eqref{main term K}. Let $\ell=4$ for quartic Dirichlet $L$-functions with characters of genus $g$, we have
	\begin{equation*}
		\begin{split}
			&M^{\mathrm{K}}_{4}\left(\Phi, \frac{2g}{3} \right)
			=\\
			&\frac{3}{g}\sum_{1\le n\le N/4}\hat{\Phi}\left(\frac{6n}{g}\right)\sum_{\substack{Q\in \mathcal{P}_{q,n/r}\\r\ge1}}\frac{d(Q)}{|Q|^{2r}\left(1+2|Q|^{-1}\right)}
			+ \frac{3s_2}{g}\sum_{1\le n\le N/4}\hat{\Phi}\left(\frac{6n}{g}\right)\sum_{\substack{Q\in \mathcal{P}_{q, n/r}\\r\ge1}}\frac{2d(Q)^2|Q|^{-1}}{|Q|^{2r}\left(1+2|Q|^{-1}\right)^2}\\
			&+O\left( 
			q^{-N/4}q^{-2G/3}q^{\epsilon g} 
			\right),
		\end{split}
	\end{equation*} 
	where $s_2$ is an explicit constant defined in \eqref{l=4 constant s_2} and $G=2g/3+1$.
\end{lemma}
\begin{proof}
	Recall that for $\ell=4$, the main term of the one-level density \eqref{main term K} comes from when $f$ is a $4^{th}$ power. Since $\chi_F(f)=\overline{\chi_F(f)}=1$, we have 
	$$M^{\mathrm{K}}_{4}\left(\Phi, \frac{2g}{3} \right)=\frac{3}{g|\mathcal{C}_4^\mathrm{K}(g)|}\sum_{1\le n\le N/4}\hat{\Phi}\left(\frac{6n}{g}\right)\sum_{\substack{Q\in \mathcal{P}_{q, n/r}\\r\ge1}}\frac{d(Q)\mathcal{K}_1}{|Q|^{2r}}.$$
	Let $$D_f(u)=\prod_{\substack{P\in \mathcal{P}_q\\P\mid f}}\left( 1+2u^{d(P)} \right)^{-1},$$
	then
	$$J(u, u)=J_0(u)D_f(u).$$
	Furthermore, if $f=Q^{4r}$ for some monic irreducible polynomial $Q$ and some integer $r$, then 
	$$D_f(u)=D_Q(u)=\left( 1+2u^{d(Q)} \right)^{-1}.$$
	Now using Lemma \ref{lemma 1 l=4 K} for $\mathcal{K}_1$, we can write the sum over $Q$ as two sums 
	$$H_1+H_2+O\left( q^{-n}q^{G/3+\epsilon g} \right)$$
	grouped by the same order of derivation on $D_Q\left(u\right)$. Here $G=2g/3+1$ as given in \eqref{Gv notations}.\\
	We have
	\begin{equation*}
		\begin{split}
			H_1
			=\frac{q^{G+2}}{2}\left[(G+1)J_0\left(\frac{1}{q}\right) - \frac{\frac{d}{du}J_0\left(u\right)\mid_{1/q}}{q} \right]  \sum_{\substack{Q\in \mathcal{P}_{q, n/r}\\r\ge1}}\frac{d(Q)}{|Q|^{2r}\left(1+2|Q|^{-1}\right)},
		\end{split}
	\end{equation*}
	and
	\begin{equation*}
		\begin{split}
			H_2=
			\frac{q^{G+2}J_0(\frac{1}{q})}{2}\sum_{\substack{Q\in \mathcal{P}_{q, n/r}\\r\ge1}}\frac{2d(Q)^2|Q|^{-1}}{|Q|^{2r}\left(1+2|Q|^{-1}\right)^2}.
		\end{split}
	\end{equation*}
	Let $s_i$ denote the constant obtained by the coefficient of $H_i$ divided by  $\left| \mathcal{C}^{\mathrm{K}}_{4}(g) \right|$.\\
	We note that $s_1=1$, and we have
	\begin{equation}
		\label{l=4 constant s_2}
		s_2=\frac{q^{G+2}J_0(\frac{1}{q})}{2\left|\mathcal{C}_4^\mathrm{K}(g) \right|}.
	\end{equation}
	Therefore
	\begin{equation*}
		\begin{split}
			&M^{\mathrm{K}}_{4}\left(\Phi, \frac{2g}{3} \right)
			=\\
			&\frac{3}{g}\sum_{1\le n\le N/4}\hat{\Phi}\left(\frac{6n}{g}\right)\sum_{\substack{Q\in \mathcal{P}_{q,n/r}\\r\ge1}}\frac{d(Q)}{|Q|^{2r}\left(1+2|Q|^{-1}\right)}
			+ \frac{3s_2}{g}\sum_{1\le n\le N/4}\hat{\Phi}\left(\frac{6n}{g}\right)\sum_{\substack{Q\in \mathcal{P}_{q, n/r}\\r\ge1}}\frac{2d(Q)^2|Q|^{-1}}{|Q|^{2r}\left(1+2|Q|^{-1}\right)^2}\\
			&+O\left( 
			q^{-N/4}q^{-2G/3}q^{\epsilon g} 
			\right),
		\end{split}
	\end{equation*} 
	where $s_2$ is of order $1/g$.
\end{proof}
\subsubsection{The Error Term}
Recall that if $f$ is a non-$4^{th}$ power, the error term contribution to the one-level density of quartic Dirichlet $L$-functions over function fields is 
$$E^{\mathrm{K}}_{4}\left(\Phi,\frac{2g}{3}\right)
=
\frac{3}{2g\left|\mathcal{C}^{\mathrm{K}}_{4}(g)\right|}\sum_{1\le n\le N}\hat{\Phi}\left(\frac{3n}{2g}\right)\sum_{\substack{f\in \mathcal{M}_{q,n}\\f \ \text{non-$4^{th}$ powers}}}\frac{\Lambda(f)}{|f|^{1/2}}{\sum_{F\in\mathcal{C}_4^\mathrm{K}(g)}}\left(\chi_F(f)+\overline{\chi_F(f)}\right).$$ 
We prove the following upper bound.
\begin{lemma}
	\label{lemma 3 l=4}
	Let $f\in\mathbb{F}_q[t]$ be a monic polynomial and $\chi_{4}$ as defined in \eqref{equation chi_ell}. Let
	$$\mathcal{K}_{2}
	=
	\sum_{ \substack{\chi \ \text{primitive quartic}\\ \chi^2 \ \text{primitive}\\ \text{genus}(\chi)=g\\ \chi\mid_{\mathbb{F}_q^*}=\chi_4} }\chi(f).$$
	Then 
	\begin{equation*}
		\begin{split}
			\mathcal{K}_{2}
			=
			\sum_{ \substack{d_1+d_3=G \\ d_1+3d_3\equiv 1 \Mod{4}}} \sum_{F_1\in \mathcal{H}_{q,d_1}}\chi_f(F_1) \sum_{ \substack{F_3\in \mathcal{H}_{q,d_3} \\ (F_3,F_1)=1}}\chi_f(F_3)^3
			\ll
			Gq^{G/2}q^{\epsilon d(f)},
		\end{split}
	\end{equation*}
	for $G=2g/3+1$.
	
	The error term for the one-level density of zeros of Dirichlet $L$-functions in the Kummer setting is given in \eqref{error term K}. For $\ell=4$, we have
	$$E^{\mathrm{K}}_{4}\left( \Phi, \frac{2g}{3} \right) 
	\ll
	q^{N/2}q^{-G/2}q^{\epsilon N}.
	$$  
\end{lemma}

\begin{proof}
	We consider the generating series of $\mathcal{K}_{2}$
	$$\mathcal{S}(u_1, u_3)
	=
	\sum_{F_1\in \mathcal{H}_{q}}\chi_f(F_1)u_1^{d(F_1)} \sum_{ \substack{F_3\in \mathcal{H}_{q} \\ (F_3,F_1)=1}}\chi_f(F_3)^3u_3^{d(F_3)}.$$
	
	First, the sum over $F_3$ can be written as the product
	$$\frac{ \prod_{\substack{P\in \mathcal{P}_q}} \left( 1+\chi_f(P)^3u_3^{d(P)} \right)} {\prod_{P \mid F_1}\left( 1+\chi_f(P)^3u_3^{d(P)} \right)},$$
	thus 
	$$\mathcal{S}(u_1, u_3)=\prod_{\substack{P\in \mathcal{P}_q}} \left( 1+\chi_f(P)^3u_3^{d(P)} \right)
	\sum_{F_1\in \mathcal{H}_{q}}\frac{\chi_f(F_1)u_1^{d(F_1)}}{\prod_{P \mid F_1}\left( 1+\chi_f(P)^3u_3^{d(P)} \right)}.$$
	Writing the sum over $F_1$ as the product as well, we combine the two and obtain
	\begin{equation*}
		\mathcal{S}(u_1, u_3)=
		\prod_{P \in \mathcal{P}_q}\left( 1+\chi_f(P)u_1^{d(P)}+\chi_f(P)^3u_3^{d(P)} \right).
	\end{equation*}
	Thus
	\begin{equation*}
		\begin{split}
			\mathcal{S}(u_1, u_3)
			=
			\frac{\mathcal{L} \left(u_1,\chi_f\right)\mathcal{L} \left(u_3,\chi_f^3\right)} {\mathcal{L}\left(u_1^2, \chi_f^2\right)\mathcal{L}\left(u_3^2,\chi_f^2\right)} 
			\prod_{P \in \mathcal{P}_q}\frac{1+\chi_f(P)u_1^{d(P)}+\chi_f(P)^3u_3^{d(P)}}{\left( 1+\chi_f(P)u_1^{d(P)}\right)\left(1+\chi_f(P)^3u_3^{d(P)} \right)},
		\end{split}
	\end{equation*}
	where the product over $P$ is absolutely convergent for $|u_1|,|u_3|<q^{-1/2}.$
	
	Using Perron's formula,
	\begin{equation*}
		\begin{split}
			\mathcal{K}_{2}
			=
			\sum_{ \substack{d_1+d_3=G \\ d_1+3d_3\equiv 1 \Mod{4} }}\frac{1}{\left(2\pi i\right)^2}\oint_{|u_1|=q^{-1/2}}  \oint_{|u_3|=q^{-1/2}} \frac{\mathcal{S}(u_1,u_3)}{u_1^{d_1}u_3^{d_3}}\frac{du_1}{u_1}\frac{du_3}{u_3}.
		\end{split}
	\end{equation*}
	Then, we use the Lindel\"{o}f hypothesis (Lemma \ref{Lindelof upper bound} and Lemma \ref{Lindelof lower bound}) to obtain a bound for each of the following integrals.\\ 
	For $\beta\in \{1,3\}$, we have
	\begin{equation*}
		\frac{1}{2\pi i}\oint_{|u_\beta|=q^{-1/2}} \frac{\mathcal{L}\left( u_\beta, \chi_f^\beta\right)}{\mathcal{L}\left( u_\beta^2, \chi_f^{2\beta}\right)u_\beta^{d_\beta}}\frac{du_\beta}{u_\beta}
		\ll
		q^{d_\beta/2}q^{\epsilon d(f)},
	\end{equation*}
	thus trivially bounding the outer sum
	$$\mathcal{K}_{2}
	\ll
	Gq^{G/2}q^{\epsilon d(f)}.
	$$
	Since Lemma \ref{Lindelof upper bound} and Lemma \ref{Lindelof lower bound} hold for $\mathcal{L}\left(u_\beta, \overline{\chi_f}^\beta\right)$ and $\mathcal{L}\left(u_\beta^2, \overline{\chi_f}^{2\beta}\right)$ for $\beta\in \{1,3\}$, we have
	$$\overline{\mathcal{K}_2}
	\ll
	Gq^{G/2}q^{\epsilon d(f)}.
	$$
	
	The error term in the Kummer setting is defined as in \eqref{error term K} for $\ell=4$.
	\begin{equation*}
		E^{\mathrm{K}}_{4}\left(\Phi,\frac{2g}{3}\right)
		=
		\frac{3}{2g\left|\mathcal{C}^{\mathrm{K}}_{4}(g)\right|}\sum_{1\le n\le N}\hat{\Phi}\left(\frac{3n}{2g}\right)\sum_{\substack{f\in \mathcal{M}_{q,n}\\f \ \text{non-$4^{th}$ powers}}}\frac{\Lambda(f)}{|f|^{1/2}}\left(\mathcal{K}_2+\overline{\mathcal{K}_2} \right).
	\end{equation*}
	We trivially bound the double sum over $f$ and $n$, then divide by $\left| \mathcal{C}^{\mathrm{K}}_{4}(g) \right|$ to obtain
	\begin{equation*}
		E^{\mathrm{K}}_{4}\left( \Phi, g \right)
		\ll
		q^{N/2}q^{-G/2}q^{\epsilon N}.
	\end{equation*}
\end{proof}
\subsubsection{Proofs of the Kummer setting results for $\ell=4$}
\begin{proof}[Proof of Theorem \ref{Theorem K2}]
	We compute the one-level density of zeros of quartic Dirichlet $L$-functions in the Kummer setting. We use some notations in Section \ref{notations} and recall that $\mathcal{C}_4^\mathrm{K}\left(g\right)$ denotes the family of quartic Dirichlet $L$-functions in the Kummer setting as in \ref{notation K fam}. We have
	\begin{equation}\label{equation 1LD K4}
		\Sigma_4^\mathrm{K}\left(\Phi, g\right)
		=
		\hat{\Phi}(0)-\mathcal{A}_{4}^\mathrm{K}\left(\Phi, \frac{2g}{3}\right),
	\end{equation}
	where
	\begin{equation*}
		\mathcal{A}^{\mathrm{K}}_{4}\left(\Phi, \frac{2g}{3}\right)
		=\frac{3}{2g|\mathcal{C}^{\mathrm{K}}_4(g)|}\sum_{1\le n\le N}\hat{\Phi}\left(\frac{3n}{2g}\right)\sum_{f\in \mathcal{M}_{q,n}}\frac{\Lambda(f)}{|f|^{1/2}}{\sum_{\substack{F\in\mathcal{C}_4^\mathrm{K}\left(g\right)\\ }}}\left[\chi_F(f) + \overline{\chi_F(f)} \right].
	\end{equation*}
	Note that $\Sigma_4^\mathrm{K}\left(\Phi, g\right)$ and $\mathcal{A}^{\mathrm{K}}_{4}\left(\Phi, \frac{2g}{3}\right)$ are given by setting $\ell=4$ in \eqref{1LD K} and \eqref{equation A Kummer} respectively.
	
	Using Lemma \ref{lemma 2 l=4} and Lemma \ref{lemma 3 l=4}, we obtain the following result in \eqref{equation final l=4}. 
	\begin{equation*}
		\begin{split}
			&\Sigma_4^\mathrm{K}\left(\Phi, g\right)
			=
			\hat{\Phi}\left(0\right)\\
			&-
			\frac{3}{g}\sum_{1\le n\le N/4}\hat{\Phi}\left(\frac{6n}{g}\right)\sum_{\substack{Q\in \mathcal{P}_{q,n/r}\\r\ge1}}\frac{d(Q)}{|Q|^{2r}\left(1+2|Q|^{-1}\right)}
			- 
			\frac{3s_2}{g}\sum_{1\le n\le N/4}\hat{\Phi}\left(\frac{6n}{g}\right)\sum_{\substack{Q\in \mathcal{P}_{q, n/r}\\r\ge1}}\frac{2d(Q)^2|Q|^{-1}}{|Q|^{2r}\left(1+2|Q|^{-1}\right)^2}\\
			&+O\left(  q^{N/2}q^{-G/2}q^{\epsilon N}\right).
		\end{split}
	\end{equation*}
	where $s_2$ is an explicit constant defined in \eqref{l=4 constant s_2} and $G=\frac{2g}{3}+1$.
\end{proof}

Using the theorem, we confirm the symmetry type of the family.
\begin{proof} [Proof of corollary 2.4 for $\ell=4$ in the Kummer setting]
	Let $N<2g/3$. Then 
	$$\lim_{g\rightarrow \infty}\Sigma_4^\mathrm{K}\left(\Phi, g\right)
	=\hat{\Phi}(0),
	$$	
	since the double sums over $n$ and $Q$ are $o(1)$ as $g\rightarrow \infty.$
	
	Furthermore, by computing the two integrals, we confirm that 
	$$
	\hat{\Phi}(0)
	=
	\int_{-\infty}^{\infty}\hat{\Phi}(y)\hat{\mathcal{W}}_{U(2g/3)}(y)dy
	=
	\int_{-\infty}^{\infty}\hat{\Phi}(y)\delta_0(y)dy,$$
	where $\mathcal{W}_{U(2g/3)}(y)=\delta_0(y)$ denotes the one-level scaling density of the group of unitary matrices. 
\end{proof}
\noindent
Thus the symmetry type of the family is unitary. This supports Katz and Sarnak's philosophy.
\\


\section{The non-Kummer setting} \label{section 5}
\subsection{The Main Term}
First we compute the main term in $\eqref{main term nK}$ of the one-level density of zeros of Dirichlet $L$-functions for $\ell=3, 4$ and $6$ in the non-Kummer setting.

\begin{lemma}
	\label{lemma main term nK}
	Let $D(\ell)$ be the degree of conductors of primitive order $\ell$ characters over $\mathbb{F}_q[t]$ as seen in \eqref{equation D(l)} for $\ell=3,4$ or 6. For a fixed $f\in \mathcal{M}_{q, n}$, we have 
	\begin{equation*}
		\begin{split}
			\sum_{\substack{F\in \mathcal{H}_{q^2, D\left(\ell \right)/2}\\P|F \Rightarrow P \notin \mathbb{F}_q[t]\\(F,f)=1}}1
			=q^{D\left(\ell\right)}\left(q^{-2}-q^{-4}\right)\frac{E_{0}\left(1/q^2\right)}{E_{f}\left(1/q^2\right)} \prod_{\substack{\pi\in \mathcal{P}_{q^2}\\\pi|f}}\left(1+|\pi|_{q^2}^{-1}\right)^{-1}+O\left(q^{D\left(\ell \right)/2+\epsilon g}\right),
		\end{split}
	\end{equation*}
	where
	\begin{equation}
		\label{equation define E_h(u)}
		E_{h}(u)=\prod_{\substack{P\in \mathcal{P}_{q}\\P\mid  h\\d(P)\equiv 0\Mod 2}}\left(1-\frac{u^{d(P)}}{(1+u^{d(P)/2})^2}\right)
		\prod_{\substack{P\in \mathcal{P}_{q}\\P\mid  h\\d(P)\equiv 1\Mod 2}}\left(1-\frac{u^{d(P)}}{1+u^{d(P)}}\right),
	\end{equation} 
	and $E_{0}(u)$ is the product over all primes in $\mathcal{P}_q$.
\end{lemma}

\begin{proof}
	Since
	$$\sum_{\substack{D\in \mathbb{F}_q[t]\\D|F}}\mu(D)=\begin{cases}
		1 &\text{if} \  F \  \text{has no prime divisors in} \ \mathbb{F}_q[t]\\
		0 &\text{otherwise}
	\end{cases},$$
	the generating series for the sum over $F$ can be written as
	\begin{equation}
		\label{equation generating series nK}
		\mathcal{S}(u)
		=\sum_{\substack{F\in \mathcal{H}_{q^2}\\P|F \Rightarrow P \notin \mathbb{F}_q[t]\\(F,f)=1}}u^{d(F)}
		=\sum_{\substack{F\in \mathcal{H}_{q^2}\\(F,f)=1}}u^{d(F)}\sum_{\substack{D\in \mathbb{F}_q[t]\\D\mid F}}\mu(D)
		=\sum_{\substack{D\in \mathbb{F}_q[t]\\(D,f)=1}}\mu(D)u^{d(D)}\sum_{\substack{F\in \mathcal{H}_{q^2}\\(F,Df)=1}}u^{d(F)}.
	\end{equation}
	Writing the inner sum over $F$ as the product over primes,  
	$$\sum_{\substack{F\in \mathcal{H}_{q^2}\\(F,Df)=1}}u^{d(F)}=\prod_{\substack{\pi\in \mathcal{P}_{q^2}\\ \pi\nmid Df}}\left(1+u^{d(\pi)}\right)
	=\frac{Z_{q^2}(u)}{Z_{q^2}(u^2)\prod_{\substack{\pi\in \mathcal{P}_{q^2}\\\pi\mid D f}}\left(1+u^{d(\pi)}\right)}.$$
	Thus $\eqref{equation generating series nK}$ above is
	$$\mathcal{S}(u)=\frac{1-q^2u^2}{(1-q^2u)\prod_{\substack{\pi\in \mathcal{P}_{q^2}\\\pi|  f}}\left(1+u^{d(\pi)}\right)}\sum_{\substack{D\in \mathbb{F}_q[t]\\(D,f)=1}}\frac{\mu(D)u^{d(D)}}{\prod_{\substack{\pi\in \mathcal{P}_{q^2}\\\pi| D}}\left(1+u^{d(\pi)}\right)}.$$
	Similarly, the sum over $D$ can be written as the product
	\begin{equation}
		\label{equation Euler product mod 2}
		\begin{split}
			\sum_{\substack{D\in \mathbb{F}_q[t]\\(D,f)=1}}\frac{\mu(D)u^{d(D)}}{\prod_{\substack{\pi\in \mathcal{P}_{q^2}\\\pi| D}}\left(1+u^{d(\pi)}\right)}
			&=\prod_{\substack{P\in \mathcal{P}_{q}\\P \nmid  f}}\left(1-\frac{u^{d(P)}}{\prod_{\substack{\pi\in \mathcal{P}_{q^2}\\\pi|P}} \left(1+u^{d(\pi)}\right)}\right)    \\
			&=\prod_{\substack{P\in \mathcal{P}_{q}\\P\nmid  f\\d(P) \equiv 0\Mod 2}}\left(1-\frac{u^{d(P)}}{(1+u^{d(P)/2})^2}\right)\prod_{\substack{P\in \mathcal{P}_{q}\\P \nmid  f\\d(P)\equiv 1\Mod2}}\left(1-\frac{u^{d(P)}}{1+u^{d(P)}}\right),
		\end{split}
	\end{equation}
	where the last equality follows from lemma 2.9 in \cite{BSM}.
	We denote this product by $\displaystyle \frac{E_{0}(u)}{E_{f}(u)}$ where $E_h(u)$ is as defined in \eqref{equation define E_h(u)}. 
	
	Note that $\displaystyle \frac{E_{0}(u)}{E_{f}(u)}$ is absolutely convergent for $|u|<1/q$. We can see this by expanding the denominator of the fraction in each product of $\eqref{equation Euler product mod 2}$, where for $j\in \{0,1\}$, each term can be written as 
	$$\displaystyle \prod_{\substack{P\in \mathcal{P}_{q}, \ P \nmid f\\d(P)\equiv j \Mod{2}}} \left(1-u^{d(P)} \ + B\left(u\right)\right)$$ where $B(u)$ contains $u^\alpha$ for $\alpha>d(P)$.\\
	Hence
	$$ \mathcal{S}(u)=\frac{1-q^2u^2}{(1-q^2u)\prod_{\substack{\pi\in \mathcal{P}_{q^2}\\\pi \mid  f}}\left(1+u^{d(\pi)}\right)}\times
	\frac{E_{0}(u)}{E_{f}(u)},
	$$
	which is absolutely convergent for $|u|<1/q^2$.
	
	Using Perron's formula (Lemma $\ref{lemma Perron's formula}$), we first integrate along a circle of radius $|u|=1/q^{2+\epsilon}$ and then shift the contour to $|u|=1/q^{1+\epsilon}$, where we encounter a simple pole at $u=1/q^2$.\\	
	Let $\mathcal{S}(u)(1-q^2u)=F(u)$. For a small circle around the origin, for example, $|u|=1/q^{100}$, we have  
	\begin{equation*}
		\begin{split}
			\sum_{\substack{F\in \mathcal{H}_{q^2, D\left(\ell \right)/2}\\P|F \Rightarrow P \notin \mathbb{F}_q[t]\\(F,f)=1}}1
			=&
			\frac{1}{2\pi i}\oint_{|u|=1/q^{100}} \frac{F(u)}{u^{D\left(\ell \right)/2}(1-q^2u)}\frac{du}{u} \\
			=&
			-\text{Res}\left(u=1/q^2\right)+\frac{1}{2\pi i}\oint_{|u|=1/q^{1+\epsilon}} \frac{F(u)}{u^{D\left(\ell \right)/2}(1-q^2u)}\frac{du}{u}.
		\end{split}
	\end{equation*}
	We bound the integral and compute the residue to obtain
	\begin{equation*}
		\begin{split}
			\sum_{\substack{F\in \mathcal{H}_{q^{2}, D\left(\ell \right)/2}\\P|F \Rightarrow P \notin \mathbb{F}_q[t]\\(F,f)=1}}1
			=&-\lim_{u\rightarrow1/q^{2}}\frac{F(u)(u-1/q^{2})}{u^{D\left(\ell \right)/2}(1-q^{2}u)}+O\left(q^{D\left(\ell \right)/2+\epsilon D\left(\ell \right)/2}\right)\\
			=&\frac{F(1/q^{2})q^{D(\ell)}}{q^{2}}+O\left(q^{D\left(\ell \right)/2+\epsilon g}\right)\\
			=& q^{D(\ell)}\left(q^{-2}-q^{-4}\right)\frac{E_{0}(1/q^{2})}{E_{f}(1/q^{2})}\prod_{\substack{\pi\in \mathcal{P}_{q^2}\\\pi \mid f}}\left(1+|\pi|_{q^2}^{-1}\right)^{-1}+O\left(q^{D\left(\ell \right)/2+\epsilon g}\right).
		\end{split}
	\end{equation*}
	This gives the result stated in Lemma \ref{lemma main term nK}.
\end{proof}


Now we compute the size of the family $|\mathcal{C}^{\mathrm{nK}}_\ell(g)|$ in the non-Kummer setting. The proof is similar to the previous lemma.
\begin{corollary} For $\ell=3, 4$ and 6, the size of the family of Dirichlet $L$-functions of order $\ell$ is 
	\begin{equation*}
		|\mathcal{C}^{\mathrm{nK}}_\ell(g)|
		=q^{D(\ell)}\left(q^{-2}-q^{-4}\right) E_{0}(1/q^{2})+O\left(q^{D\left(\ell \right)/2+\epsilon g}\right)
	\end{equation*}
\end{corollary}

\begin{proof}
	The number of primitive characters of order $\ell$ with conductor of degree $D(\ell)$ is given by
	\begin{equation*}
		|\mathcal{C}^{\mathrm{nK}}_\ell(g)|=\sum_{\substack{F\in \mathcal{H}_{q^2, D\left(\ell \right)/2}\\P|F \Rightarrow P \notin \mathbb{F}_q[t]}}1.
	\end{equation*}
	Using the same method as in Lemma \ref{lemma main term nK},
	its generating series can be written as
	$$\mathcal{S}(u)=\sum_{\substack{F\in \mathcal{H}_{q^2}\\P|F \Rightarrow P \notin \mathbb{F}_q[t]}}u^{d(F)}=\sum_{\substack{F\in \mathcal{H}_{q^2}}}u^{d(F)}\sum_{\substack{D\in \mathbb{F}_q[t]\\D|F}}\mu(D)
	=\sum_{\substack{D\in \mathbb{F}_q[t]}}\mu(D)u^{d(D)}\sum_{\substack{F\in \mathcal{H}_{q^2}\\(F,D)=1}}u^{d(F)}.$$
	The sum over $F$ is 
	$$\sum_{\substack{F\in \mathcal{H}_{q^2}\\(F,D)=1}}u^{d(F)}
	=\prod_{\substack{\pi\in \mathcal{P}_{q^2}\\\pi\nmid  D }}\left(1+u^{d(\pi)}\right)
	=\frac{Z_{q^2}(u)}{Z_{q^2}(u^2)\prod_{\substack{\pi\in \mathcal{P}_{q^2}\\\pi| D}}\left(1+u^{d(\pi)}\right)},$$
	so we can rewrite $\mathcal{S}(u)$ as
	$$\mathcal{S}(u)
	=
	\frac{1-q^2u^2}{1-q^2u}\sum_{\substack{D\in \mathbb{F}_q[t]}}\frac{\mu(D)u^{d(D)}}{\prod_{\substack{\pi\in \mathcal{P}_{q^2}\\\pi| D}}\left(1+u^{d(\pi)}\right)}.$$
	Now the sum over $D$ can be written as the product
	\begin{equation*}
		\begin{split}
			\sum_{\substack{D\in \mathbb{F}_q[t]}}\frac{\mu(D)u^{d(D)}}{\prod_{\substack{\pi\in \mathcal{P}_{q^2}\\\pi| D}}\left(1+u^{d(\pi)}\right)}
			&=\prod_{\substack{P\in \mathcal{P}_{q}}}\left(1-\frac{u^{d(P)}}{\prod_{\substack{\pi\in \mathcal{P}_{q^2}\\\pi|P}} \left(1+u^{d(\pi)}\right)}\right)    \\
			&=\prod_{\substack{P\in \mathcal{P}_{q}\\d(P) \equiv 0\Mod 2}}\left(1-\frac{u^{d(P)}}{(1+u^{d(P)/2})^2}\right)\prod_{\substack{P\in \mathcal{P}_{q}\\d(P)\equiv 1\Mod2}}\left(1-\frac{u^{d(P)}}{1+u^{d(P)}}\right),
		\end{split}
	\end{equation*}
	which we denote by $\displaystyle E_{0}(u)$ as in Lemma \ref{lemma main term nK}. Note that $E_{0}(u)$ is absolutely convergent for $|u|<1/q$ as in the same lemma. 
	Thus 
	$$ \mathcal{S}(u)=\frac{\left(1-q^2u^2\right)E_{0}(u)}{1-q^2u},
	$$
	which is absolutely convergent for $|u|<1/q^2$.
	
	Using Perron's formula, we first integrate along a circle of radius $|u|=1/q^{2+\epsilon}$, then we shift the contour to $|u|=1/q^{1+\epsilon}$ and encounter a simple pole at $u=1/q^2$.\\
	Let $\mathcal{S}(u)(1-q^2u)=F(u)$. We have
	\begin{equation*}
		\begin{split}
			&\sum_{\substack{F\in \mathcal{H}_{q^2, D\left(\ell \right)/2}\\P|F \Rightarrow P \notin \mathbb{F}_q[t]}}1
			=\frac{1}{2\pi i}\oint_{|u|=1/q^{100}} \frac{F(u)}{u^{D\left(\ell \right)/2}(1-q^2u)}\frac{du}{u} \\
			&=-\text{Res}\left(u=1/q^2\right)+\frac{1}{2\pi i}\oint_{|u|=1/q^{1+\epsilon}} \frac{F(u)}{u^{D\left(\ell \right)/2}(1-q^2u)}\frac{du}{u}.
		\end{split}
	\end{equation*}
	Bounding the integral and computing the residue we obtain
	\begin{equation*}
		\begin{split}
			\left|\mathcal{C}_\ell^{\mathrm{nK}}(g)\right|
			=
			\sum_{\substack{F\in \mathcal{H}_{q^{2}, D\left(\ell \right)/2}\\P|F \Rightarrow P \notin \mathbb{F}_q[t]}}1
			=&-\lim_{u\rightarrow1/q^{2}}\frac{F(u)(u-1/q^{2})}{u^{D\left(\ell \right)/2}(1-q^{2}u)}+O\left(q^{D\left(\ell \right)/2+\epsilon D\left(\ell \right)/2}\right)\\
			=&\frac{F(1/q^{2})q^{D(\ell)}}{q^{2}}+O\left(q^{D\left(\ell \right)/2+\epsilon g}\right)\\
			=& q^{D(\ell)}\left(q^{-2}-q^{-4}\right)E_{0}\left(1/q^{2}\right)+O\left(q^{D\left(\ell \right)/2+\epsilon g}\right).
		\end{split}
	\end{equation*}
\end{proof}
Using our previous results, we compute the main term $M^{\mathrm{nK}}_{\ell}(\Phi, D\left(\ell\right)-2)$ given in \eqref{main term nK}. We recall some notations in Section \ref{notations} as needed.
\begin{lemma}
	\label{lemma main main term nK}
	Let $m_Q=\gcd\left(2, d(Q)\right)$. The main term of the one-level density of zeros of Dirichlet $L$-functions with order $\ell$ characters for $\ell=3,4$ and $6$ in the non-Kummer setting as given in \eqref{main term nK} is
	\begin{equation*}
		\begin{split}
			&M_\ell^{\mathrm{nK}}\left(\Phi, D\left(\ell\right)-2\right)=\\
			&\frac{\ell-1}{g}
			\sum_{1\le n\le N/\ell}\hat{\Phi}\left(\frac{\ell(\ell-1)n}{2g}\right) \sum_{\substack{Q\in \mathcal{P}_{q, n/r}\\r \ge 1}}\frac{d(Q)}{|Q|^{\ell r/2}\left(1+|Q|^{-{2/m_Q}}\right)^{m_Q}}
			+O\left(q^{{-D\left(\ell \right)/2}+\epsilon g}\right).
		\end{split}	
	\end{equation*}
\end{lemma}
\begin{proof}
	Recall that we use $\mathcal{C}_\ell^\mathrm{nK}\left(g\right)$ as in \ref{notation nK fam} to denote the family of order $\ell$ Dirichlet $L$-functions in the non-Kummer setting. In this case, the number of non-trivial zeros is $\displaystyle D\left(\ell\right)-2=\frac{2g}{\ell-1}$. Thus the main term, which comes from the $\ell^{th}$ power polynomials, is $M_\ell^{\mathrm{nK}}\left(\Phi, \frac{2g}{\ell-1}\right)$ as follows.
	\begin{equation}
		\label{equation of main term lemma}
		\begin{split}
			M^{\mathrm{nK}}_{\ell}\left(\Phi, \frac{2g}{\ell-1}\right)=
			\frac{\ell-1}{2g|\mathcal{C}_\ell^\mathrm{nK}(g)|}\sum_{1\le n\le N/\ell}\hat{\Phi}\left(\frac{\ell(\ell-1)n}{2g}\right)\sum_{\substack{Q\in \mathcal{P}_{q, n/r}\\r\ge 1}}\frac{d(Q)}{|Q|^{\ell r/2}}\sum_{\substack{F\in \mathcal{H}_{q^{2}, D\left(\ell \right)/2}\\P|F \Rightarrow P \notin \mathbb{F}_q[t]}}\chi_F(Q^{\ell r})+\overline{\chi_F(Q^{\ell r})}.
		\end{split}
	\end{equation}
	Since $\chi_F(Q^{\ell r})=\overline{\chi_F(Q^{\ell r})}=1$, we have
	\begin{equation}
		\label{equation of main term lemma}
		\begin{split}
			M^{\mathrm{nK}}_{\ell}\left(\Phi, \frac{2g}{\ell-1}\right)=
			\frac{\ell-1}{g|\mathcal{C}_\ell^\mathrm{nK}(g)|}\sum_{1\le n\le N/\ell}\hat{\Phi}\left(\frac{\ell(\ell-1)n}{2g}\right)\sum_{\substack{Q\in \mathcal{P}_{q, n/r}\\r\ge 1}}\frac{d(Q)}{|Q|^{\ell r/2}}\sum_{\substack{F\in \mathcal{H}_{q^{2}, D\left(\ell \right)/2}\\P|F \Rightarrow P \notin \mathbb{F}_q[t]}}1.
		\end{split}
	\end{equation}
	Using Lemma \ref{lemma main term nK}, the sum over $Q$ can be written as
	\begin{equation}
		\label{equation sum over Q}
		\begin{split}
			q^{D(\ell)}\left(q^{-2}-q^{-4}\right)		 
			\sum_{\substack{Q\in \mathcal{P}_{q, n/r}\\r\ge 1}}\frac{d(Q)E_{0}(1/q^{2})}{|Q|^{\ell r/2}E_{f}(1/q^{2})}\prod_{\substack{\pi\in \mathbb{F}_{q^{2}}[t]\\\pi|Q}}\left(1+|\pi|_{q^{2}}^{-1}\right)^{-1}
			+O\left({q^{{D\left(\ell \right)/2}+\epsilon g}}\right)\\
			=
			q^{D(\ell)}\left(q^{-2}-q^{-4}\right)		 
			\sum_{\substack{Q\in \mathcal{P}_{q, n/r}\\r \ge 1}}\frac{d(Q)E_{0}(1/q^{2})}{|Q|^{\ell r/2}E_{f}(1/q^{2})\left(1+|Q|^{-{2/m_Q}}\right)^{m_Q}}
			+O\left({q^{{D\left(\ell \right)/2}+\epsilon g}}\right).
		\end{split}
	\end{equation}
	We divide \eqref{equation sum over Q} by $|\mathcal{C}_\ell^\mathrm{nK}(g)|$ and obtain	
	\begin{equation*}
		\sum_{\substack{Q\in \mathcal{P}_{q, n/r}\\r \ge 1}}\frac{d(Q)}{|Q|^{\ell r/2}\left(1+|Q|^{-{2/m_Q}}\right)^{m_Q}}+
		O\left(\frac{Nq^{{-D\left(\ell \right)/2}+\epsilon g}}{g}\right).
	\end{equation*}
	Note here that the error term from dividing the size of the family is of the same size as the error term above.
	Now using the equation above and \eqref{equation of main term lemma} gives the stated result.
\end{proof}
\subsection{The Error Term}
Recall that the error term in \eqref{error term nK} comes from the non-$\ell^{th}$ power contributions. 
To compute the error term we first prove the following lemma.
\begin{lemma}
	\label{lemma 1 nK error term}
	Let $f$ be a monic non-$\ell^{th}$ power polynomial in $\mathbb{F}_q[t]$. Then
	$$
	\sum_{\substack{F\in \mathcal{H}_{q^{2}, D\left(\ell \right)/2 }\\P|F \Rightarrow P \notin \mathbb{F}_q[t]}}\chi_F(f) 
	\ll q^{D\left(\ell \right)/2}q^{{}\epsilon \left( d(f)+ D\left(\ell \right) \right)}.
	$$
\end{lemma}

\begin{proof}
	We first write the sum over $F$ as
	\begin{equation*}
		\begin{split}
			\sum_{\substack{F\in \mathcal{H}_{q^{2}, D\left(\ell \right)/2}\\P|F \Rightarrow P \notin \mathbb{F}_q[t]}}\chi_F(f)
			&=
			\sum_{\substack{F\in \mathcal{H}_{q^{2}, D\left(\ell \right)/2}}}\chi_F(f)\sum_{\substack{D\in \mathbb{F}_q[t]\\D|F}}\mu(D)\\
			&=
			\sum_{\substack{D\in \mathbb{F}_q[t]\\d(D)\le D\left(\ell \right)/2\\(D,f)=1}}\mu(D)\sum_{\substack{F\in \mathcal{H}_{q^{2}, D\left(\ell \right)/2-d(D)}\\(F,D)=1}}\chi_F(f),
		\end{split}
	\end{equation*}
	where the inner sum over $F$ has the generating series
	$$\sum_{\substack{F\in \mathcal{H}_{q^{2}}}}\chi_F(f)u^{d(F)}
	=
	\prod_{\substack{\pi\in \mathcal{P}_{q^2}\\\pi \nmid Df}}1+\chi_{\pi}(f)u^{d(\pi)}
	=
	\frac{\mathcal{L}_{q^{2}}(u, \chi_{f})}{\mathcal{L}_{q^{2}}(u^2, \chi_{f}^2)}
	\prod_{\substack{\pi\in \mathcal{P}_{q^2}\\ \pi\mid D\\ \pi\nmid  f}}\frac{1-\chi_\pi(f)u^{d(\pi)}}{1-\chi_\pi^2(f)u^{2d(\pi)}}.$$
	We evaluate the inner sum above using Perron's formula. Thus
	$$\sum_{\substack{F\in \mathcal{H}_{q^{2}, D\left(\ell \right)/2-d(D)}\\(F,D)=1}}\chi_F(f)=\frac{1}{2\pi i}\oint\frac{\mathcal{L}_{q^{2}}(u, \chi_{f})}{\mathcal{L}_{q^{2}}(u^2, \chi_{f}^2)u^{D\left(\ell \right)/2-d(D)}}
	\prod_{\substack{\pi\in \mathcal{P}_{q^{2}}\\\pi\mid D\\ \pi\nmid  f}}\frac{1-\chi_{\pi}(f)u^{d(\pi)}}{1-\chi_{\pi}^2(f)u^{2d({\pi})}}\frac{du}{u},$$
	where we integrate along a circle of radius $|u|=q^{-{1}}$ around the origin. 
	
	The Lindel\"{o}f bound in Lemma \ref{Lindelof upper bound} for the $L$-function in the numerator gives
	$$|\mathcal{L}_{q^{2}}(u, \chi_{f})|\ll q^{2\epsilon d(f)},$$ 
	and the lower bound in Lemma \ref{Lindelof lower bound} for the $L$-function in the denominator gives
	$$|\mathcal{L}_{q^{2}}(u^2, \chi_{f}^2)|\gg q^{-2\epsilon d(f)}.$$ 
	Hence
	$$\sum_{\substack{F\in \mathcal{H}_{q^{2}, D\left(\ell \right)/2-d(D)}\\(F,D)=1}}\chi_F(f)
	\ll
	q^{D\left(\ell \right)/2-d(D)}q^{{4}\epsilon d(f)+2\epsilon d(D)}.$$
	Now, trivially bounding the sum over $D$ gives
	$$\sum_{\substack{F\in \mathcal{H}_{q^{2}, D\left(\ell \right)/2}\\P|F \Rightarrow P \notin \mathbb{F}_q[t]}}\chi_F(f)
	\ll\sum_{m=0}^{D\left(\ell \right)/2}\sum_{\substack{D\in \mathbb{F}_q[t]\\d(D)=m}}q^{D\left(\ell \right)/2-m}q^{{4}\epsilon d(f)+2\epsilon m}
	\ll q^{D\left(\ell \right)/2}q^{{}\epsilon\left(d(f) + D\left(\ell \right)\right)  }.
	$$
	This concludes the proof of the above lemma.
\end{proof}

Using Lemma \ref{lemma 1 nK error term}, we bound the contribution from $\displaystyle E_\ell^{\mathrm{nK}}\left( \Phi, D(\ell)-2 \right)$.
\begin{lemma}
	\label{lemma error term nK}
	The contribution of non-$\ell^{th}$ power terms is bounded by
	\begin{equation*}
		\begin{split}
			E_\ell^{\mathrm{nK}}\left(\Phi, D(\ell)-2\right)
			\ll
			q^{N/2}q^{-D\left(\ell \right)/2}q^{\epsilon(N+g)}.
		\end{split}
	\end{equation*}
\end{lemma}

\begin{proof}
	Recall that 
	\begin{equation*}
		\begin{split}
			&E^{\mathrm{nK}}_{\ell}(\Phi,D\left(\ell\right)-2)\\
			&=
			\frac{1}{(D\left(\ell\right)-2)\left|\mathcal{C}^{\mathrm{nK}}_\ell\left(g\right)\right|}\sum_{1\le n\le N}\hat{\Phi}\left(\frac{n}{D\left(\ell\right)-2}\right)\sum_{\substack{f\in \mathcal{M}_{q, n}\\f \ \text{non-$\ell^{th}$ power}}}\frac{\Lambda(f)}{|f|^{1/2}}\sum_{\substack{F\in \mathcal{H}_{q^2, D\left(\ell \right)/2}\\P|F \Rightarrow P \notin \mathbb{F}_q[t]}}\left[\chi_F(f)+\overline{\chi_F(f)}\right],
		\end{split}
	\end{equation*}
	as in \eqref{error term nK}.\\
	Lemma \ref{lemma 1 nK error term} implies that,
	$$\sum_{\substack{F\in \mathcal{H}_{q^{2}, D\left(\ell \right)/2}\\P|F \Rightarrow P \notin \mathbb{F}_q[t]}}\overline{\chi_F(f)}
	\ll
	 q^{D\left(\ell \right)/2}q^{{}\epsilon\left(d(f) + D\left(\ell \right)\right)  }.$$
	 Thus 
	\begin{equation*}
		\begin{split}
			E_\ell^{\mathrm{nK}}\left(\Phi, D(\ell)-2\right)
			&\ll
			\frac{\ell-1}{g|\mathcal{C}_\ell^\mathrm{nK}(g)|}\sum_{1\le n\le N}\hat{\Phi}\left( \frac{(\ell-1)n}{2g} \right) \sum_{\substack{f\in \mathcal{M}_n\\f \ \text{non-$\ell^{th}$ powers}}}\frac{\Lambda(f)q^{D\left(\ell \right)/2}q^{{}\epsilon\left(n + D\left(\ell \right)\right)  }}{|f|^{1/2}}
			\\
			&\ll
			\frac{\ell-1}{g|\mathcal{C}_\ell^\mathrm{nK}(g)|}\sum_{1\le n\le N}\hat{\Phi}\left( \frac{(\ell-1)n}{2g} \right) \frac{nq^n}{q^{n/2}}
			\left(q^{D\left(\ell \right)/2}q^{{}\epsilon\left(n + D\left(\ell \right)\right)  }\right)\\
			&\ll
			q^{N/2}q^{-D\left(\ell \right)/2}q^{\epsilon(N+g)},
		\end{split}
	\end{equation*}
	which gives the result above.
\end{proof}
\subsection{The Non-Kummer Setting Results}
\begin{proof}[Proof of Theorem \ref{Theorem nK}]
	Recall some notations in Section \ref{notations}. Since $\displaystyle D(\ell)-2=\frac{2g}{\ell-1}$, the one-level density in the non-Kummer setting is 
	\begin{equation}\label{equation 1LD nK}
		\Sigma_\ell^\mathrm{nK}\left(\Phi, g\right)
		=
		\hat{\Phi}(0)-\mathcal{A}_{\ell}^{\mathrm{nK}}\left(\Phi, \frac{2g}{\ell-1}\right)-\frac{\ell-1}{g}\sum_{1\le n\le N}\hat{\Phi}\left(\frac{n}{D\left(\ell\right)-2}\right)q^{-n/2},
	\end{equation}
	where as in \eqref{equation A nKummer}
	\begin{equation*}
		\begin{split}
			&\mathcal{A}^{\mathrm{nK}}_{\ell}\left(\Phi, \frac{2g}{\ell-1}\right)\\
			&=
			\frac{1}{\left(D\left(\ell\right)-2\right)\left|\mathcal{C}^{\mathrm{nK}}_\ell\left(g\right)\right|}\sum_{1\le n\le N}\hat{\Phi}\left(\frac{n}{D\left(\ell\right)-2}\right)\sum_{f\in \mathcal{M}_{q,n}}\frac{\Lambda(f)}{|f|^{1/2}}\sum_{\substack{F\in \mathcal{H}_{q^2, D\left(\ell \right)/2}\\P|F \Rightarrow P \notin \mathbb{F}_q[t]}}\left[\chi_F(f)+\overline{\chi_F(f)}\right].
		\end{split}
	\end{equation*}
	From Lemma \ref{lemma main main term nK} and Lemma \ref{lemma error term nK}, we write Equation \eqref{equation 1LD nK} above as
	\begin{equation*}
		\begin{split}
			&\Sigma_\ell^\mathrm{nK}\left(\Phi, g\right)
			=
			\hat{\Phi}(0)-\frac{\ell-1}{g}\sum_{1\le n\le N}\hat{\Phi}\left(\frac{(\ell-1)n}{2g}\right)q^{-n/2}\\
			&-\frac{\ell-1}{g}
			\sum_{1\le n\le N/\ell}\hat{\Phi}\left(\frac{\ell(\ell-1)n}{2g}\right) \sum_{\substack{Q\in \mathcal{P}_{q, n/r}\\r \ge 1}}\frac{d(Q)}{|Q|^{\ell r/2}\left(1+|Q|^{-{2/m_Q}}\right)^{m_Q}}
			+O\left(
			q^{N/2}q^{-D\left(\ell \right)/2}q^{\epsilon(N+g)}\right),
		\end{split}
	\end{equation*}
	where $m_Q=\gcd(d(Q),2).$ This is the result in Theorem \ref{Theorem nK}.
\end{proof}

As a corollary, we confirm the symmetry type of the family.
\begin{proof} [Proof of corollary 2.4 in the non-Kummer setting]
	Let $\displaystyle N<\frac{2g}{\ell-1}$. Then 
	$$\lim_{g\rightarrow \infty}\Sigma_\ell^\mathrm{nK}\left(\Phi, g\right)
	=\hat{\Phi}(0),
	$$	
	since the double sums over $n$ and $Q$ above are $o(1)$ as $g\rightarrow \infty.$
	
	Furthermore, compute both integeral to confirm that 
	$$
	\hat{\Phi}(0)
	=
	\int_{-\infty}^{\infty}\hat{\Phi}(y)\hat{\mathcal{W}}_{U(D(\ell)-2)}(y)dy
	=
	\int_{-\infty}^{\infty}\hat{\Phi}(y)\delta_0(y)dy,$$
	where $\mathcal{W}_{U(D(\ell)-2)}(y)=\delta_0(y)$ denotes the one-level scaling density of the group of unitary matrices.
\end{proof}
This proves the symmetry types of the families are unitary and it supports the philosophy of Katz and Sarnak.

\bibliographystyle{amsalpha}
\bibliography{1LD_Ref.bib}

\end{document}